\documentclass[a4paper,11pt]{article}
\usepackage[multidot]{grffile}
\usepackage{latexsym,amssymb,enumerate,amsmath,epsfig,amsthm}
\usepackage[margin=1in]{geometry}
\usepackage{setspace,color}
\usepackage{graphics}
\usepackage{graphicx}
\usepackage[ruled]{algorithm2e}
\usepackage{empheq}
\usepackage{bm,amsmath}
\usepackage{subcaption}
\usepackage{hyperref} 
\hypersetup{
    colorlinks=true,
    linkcolor=blue,
    filecolor=magenta,      
    urlcolor=cyan,
    citecolor=red,
    pdfpagemode=FullScreen,
    }

\usepackage{rotating,booktabs,array,amsmath,amssymb}
\newcolumntype{C}{>{$\displaystyle}c<{$}} % automatic display-style math mode

\usepackage{multirow}
\newcommand\fixup{\kern-\fontcharic\scriptfont2`\"}

\usepackage{lineno}

\DeclareMathOperator*{\argmin}{argmin}

\newcommand{\x}{\mathbf{x}}
\newcommand{\s}{\mathbf{s}}
\newcommand{\p}{\mathbf{p}}
\newcommand{\q}{\mathbf{q}}
\newcommand{\br}{\mathbf{r}}
\newcommand{\m}{\mathbf{m}}
\newcommand{\bx}{\mathbf{x}}

\newcommand{\bp}{\mathbf{p}}
\newcommand{\bq}{\mathbf{q}}
\newcommand{\bs}{\mathbf{s}}
\newcommand{\bk}{\mathbf{k}}

\newtheorem{thm}{Theorem}%[section]
\newtheorem{prop}{Proposition}%[section]
\title{SLERP-TVDRK (STVDRK) Methods for Ordinary Differential Equations on Spheres}
\author{
Shingyu Leung\thanks{Department of Mathematics, the Hong Kong University of Science and Technology, Clear Water Bay, Hong Kong. Email: {\bf masyleung@ust.hk}}
\and
Wai Ming Chau\thanks{Department of Mathematics, the Hong Kong University of Science and Technology, Clear Water Bay, Hong Kong. Email: {\bf wmchau@connect.ust.hk}}
\and
Young Kyu Lee\thanks{Department of Mathematics, the Hong Kong University of Science and Technology, Clear Water Bay, Hong Kong. Email: {\bf ykleeac@connect.ust.hk}}
}

\date{}

\begin{document}
\thispagestyle{plain}
\maketitle

\begin{abstract}
{We mimic the conventional explicit Total Variation Diminishing Runge-Kutta (TVDRK) schemes and propose a class of numerical integrators to solve differential equations on a unit sphere.} Our approach utilizes the exponential map inherent to the sphere and employs spherical linear interpolation (SLERP). These modified schemes, named SLERP-TVDRK methods or STVDRK, offer improved accuracy compared to typical projective RK methods. Furthermore, they eliminate the need for any projection and provide a straightforward implementation. While we have successfully constructed STVDRK schemes only up to third-order accuracy, we explain the challenges in deriving STVDRK-$r$ for $r\ge4$. To showcase the effectiveness of our approach, we will demonstrate its application in solving the eikonal equation on the unit sphere and simulating $p$-harmonic flows using our proposed method.
\end{abstract}

\section{Introduction}
\label{Sec:Introduction}

We consider solving the ordinary differential equation (ODE) on $\mathbb{S}^2$ given by $\p'(t)=f(\p(t),t)$ with an initial condition $\mathbf{p}(0) = \mathbf{p}_0 \in \mathbb{S}^2$. Here, $f:\mathbb{S}^2\times [0,\infty) \rightarrow \mathcal{T}(\mathbb{S}^2)$ satisfies a Lipschitz function, ensuring that the solution to the ODE remains confined to $\mathbb{S}^2$ for all time. This formulation finds natural application in scenarios such as path planning for rigid bodies, where the ODE of the special orthogonal group $SO(3)$ needs to be solved \cite{shi09}. Other applications include quantum field theory within quantum mechanics \cite{adl86}, protein structure modeling \cite{pro14}, molecular dynamics simulation \cite{rap85}, fluid mechanics theory \cite{ghkr06}, fluid flow visualization \cite{hanma95}, computations involving flexible filaments and fibers in complex fluids \cite{tscfro20,stwk21}, differential equations \cite{kouxia18}, and dynamics of rigid bodies \cite{weiterfed06,wil09,udwsch10}. 
%This work will discuss various applications, particularly high-frequency wave propagation on a unit sphere utilizing geometrical optics \cite{gri68} and $p$-harmonic flows for signal denoising \cite{tansapcas00,tansapcas01,vesosh02,lysoshtai04,golwenyin09}. 

Numerically, a straightforward approach to formulate a numerical scheme for the given ODE is first to extend the velocity field $f$ using the expression:
\begin{equation}
f(\mathbf{x},t) = f(\mathcal{P}(\mathbf{x}),t),
\label{Eqn:VelocityExtension}
\end{equation}
where $\mathcal{P}(\mathbf{x}) = \mathbf{x}/\|\mathbf{x}\|$ represents the projection (or the closest point) operator. The ODE is then solved within $\mathbb{R}^3$ rather than $\mathbb{S}^2$. This approach has been utilized in various instances \cite{haiwan96,eicfuh98,hai00,hai01}, and numerous high-order numerical schemes can be employed within this framework. {For the rest of this paper}, we assume that the temporal direction is discretized uniformly with time step size $h$, i.e., $h = t^{n+1} - t^n$, where the numerical integrator updates the solution $\mathbf{p}^n$ to approximate $\mathbf{p}(t^{n+1})$ (denoted as $\mathbf{p}^{n+1}$).

The main challenge associated with this ODE is that there is no guarantee that the solutions stay on $\mathbb{S}^2$ in evolutions. Taking the Runge-Kutta (RK) methods as examples, we might apply the projection operator $\mathcal{P}$ to \textit{the final stage} (but not to any intermediate stage) to enforce this constraint at all $t=t^n$. This procedure gives the following standard projection methods. Some analysis of the properties of these projection methods can be found in, for example, \cite{chmr06}. Except for the projection method based on the RK technique, there are also semi-implicit projection methods \cite{xgwzc20}, a spherical midpoint method \cite{mclmodver17}. Although the projection steps enforce the numerical solution staying on the unit sphere for all time, the resulting numerical solutions could be quite unsatisfactory in practice \cite{hai00,hai01}. Regarding the accuracy of the numerical method, we require that the error introduced by the projection step(s) has the same order as the original standard Cartesian RK method so that the extra projection step(s) would not alter the overall accuracy. This requirement might be challenging to demonstrate for a general projected-RK scheme, especially when the method involves multiple project steps. In particular, we have performed a numerical test in Section \ref{Sec:Examples}, which demonstrates that the norm of the numerical solution from the standard second-order total variation diminishing RK (TVDRK2) scheme converges to the unity with $O(h^3)$ and the PTVDRK3 method might be second-order accurate only if we apply the projection step in all the intermediate stages.

The second class of methods relies more explicitly on the structure of the manifold \cite{wehhau82,potrhe91,crogroyan92,calisezan96,dieloppel98,mun99,lewnig03,celmarowr14}. Therefore, this approach requires at least a local parameterization of the manifold to build a high-order numerical approximation to the solution. In particular, we highlight the high-order Munthe-Kaas RK method as developed in \cite{mun99}. Let $a_{i,j}$ and $b_j$ be the coefficients of an $s$-stage $q$-th order classical RK method with $c_i=\sum_j a_{i,j}$ and let $\lambda(\mathbf{u},\p)=\Lambda(\mbox{exp}(\mathbf{u}),\p)$ be a (left) Lie group action. The approach turns any classical RK method into a numerical algorithm of the same order on a general homogeneous manifold,
\begin{equation}
\left\{
\begin{array}{l}
\q_i = h \sum_{j=1}^s a_{i,j} \s'_j \\
\s_i = f(hc_i,\lambda(\q_i,\p^n)) \\
\s'_i = \mbox{dexpinv}(\q_i,\s_i, q)
\end{array}
\right.
\label{Eqn:RKonManifolds}
\end{equation}
for $i=1,\cdots,s$ and $\p^{n+1}=\lambda\left( h \sum_{j=1}^s b_j \s'_j, \p^n \right)$. The expression $\mbox{dexpinv}(\mathbf{u},\mathbf{v}, q)$ is a $q$-th order truncated approximation for $\mbox{dexp}_{\mathbf{u}}^{-1}(\mathbf{v})$ relating to the inverse of the differential of the exponential mapping. This class of RK methods constructs high-order corrections to the left Lie algebra action to the solution at the $n$-th step. Once this action is determined, it immediately applies to $\p^n$ to obtain the new update, and it requires no interpolation on the underlying manifold. This class of RK methods is similar to the standard projection methods (such as the PRK3 or PRK4) mentioned above. The significant difference is that this method directly computes the action on the manifold rather than in the embedded space, followed by a projection back to the surface.

This paper introduces a new class of numerical integrators designed to solve ODEs on the unit sphere up to third-order accuracy. We adopt the TVDRK methodology \cite{shu88,shuosh88,gotshu98,gotshutad00}, which constructs higher-order numerical solutions using convex combinations of elementary forward Euler-type building blocks. {The TVDRK method is designed to preserve the total-variation diminishing (TVD) property of the underlying ODE. This means the method ensures that the total variation of the numerical solution does not increase over time, which is crucial for solving partial differential equations (PDEs) such as hyperbolic conservation laws and other problems involving discontinuities. Treating it as typical numerical integrators for solving ODE, these methods require only the implementation of a subroutine for the standard forward-Euler update and the computations of a convex combination of several solutions from the intermediate states. This keeps the overall algorithm relatively straightforward to implement. Due to the TVDRK idea,} our integrator avoids relying on multiple corrective steps to approximate the numerical differentiation by sidestepping the need for inverse exponential maps. 

{Following the TVDRK framework, we develop two key ingredients in our proposed integrators. Unlike the conventional projected RK methods that require an additional projection step, our approach leverages the explicit formula of the exponential map on the unit sphere. This approach yields a numerical solution that automatically satisfies the $\mathbb{S}^2$ constraint. Second, we note that simple (weighted) averaging of intermediate solutions will lift the solution off from the unit sphere. Therefore}, we incorporate the concept of spherical linear interpolation (SLERP) to replicate the convex combination of two intermediate stages within conventional TVDRK methods. This construction gives rise to a class of straightforward, high-order, and efficient approximate schemes for solving ODEs on $\mathbb{S}^2$.

The rest of the paper is organized as follows. We summarize the SLERP formula and some projected TVDRK methods in Section \ref{Sec:Background}. The SLERP interpolating approach is the building block of our proposed STVDRK methods, which will be fully discussed in Section \ref{Sec:STVDRK}. Finally, we provide some examples to demonstrate the accuracy of our numerical approach in Section \ref{Sec:Examples}.

%%%%%%%%%%%%%%%%%
\section{Background}
\label{Sec:Background}

We provide the essential components on the Spherical Linear Interpolation (SLERP) formula. This foundational concept will serve as a fundamental building block within our proposed category of TVDRK methods designed for ODE integrator on the unit sphere. Subsequently, we will provide an overview of diverse existing numerical integrators based on the projection concept. 

\subsection{The Spherical Linear Interpolation (SLERP) Formula}
\label{SubSec:SLERP}

This section provides the background on the interpolation of spherical data and introduces the spherical linear interpolation (SLERP) formula \cite{shoemake_85,sola} using a quaternion representation \cite{ham63}. Quaternions are numbers consisting of four dimensions, one real part, and a three-dimensional analogy to {the imaginary part of} complex numbers. A quaternion can be written in many forms: \[\underset{\text{real}}{\boxed{a}} + \underset{\text{imaginary}}{\boxed{b \mathbf{i} + c\mathbf{j} + d\mathbf{k}}} = (a, b, c, d) = (\underset{\text{scalar}}{\boxed{a}}, \underset{\text{vector}}{\boxed{\mathbf{u}}}),\] where \(a, b, c, d \in \mathbb{R}\), $\mathbf{u} = (b, c, d) \in \mathbb{R}^3$. The notations $\mathbf{i}$, $\mathbf{j}$, and $\mathbf{k}$ are extensions of {the imaginary part of} complex numbers.

{Based on the quaternion representation, the SLERP (\textit{Spherical Linear intERPolation}) formula can be expressed by
$$
\mbox{SLERP}(\mathbf{\mathbf{q_a}}, \mathbf{q_b}, t) = (\mathbf{q_a}) ((\mathbf{q_a})^{-1} \mathbf{q_b})^t, \quad t \in [0, 1] \, , 
$$
where we have applied the following quaternion properties
\begin{itemize}
\item Hamilton product: $(a_1, \mathbf{u_1})(a_2, \mathbf{u_2}) = (a_1a_2 - \mathbf{u_1} \cdot \mathbf{u_2}, a_1\mathbf{u_2} + a_2\mathbf{u_1} + \mathbf{u_1} \times \mathbf{u_2})$ where the notation {$\cdot$ and $\times$ denotes the typical dot and cross product.}
\item Inverse map: $\mathbf{q}^{-1} = (a, -\mathbf{u})/(a^2 + b^2 + c^2 + d^2)$.
\item Exponential map: $\exp(a, \mathbf{u}) = \exp(a)(\cos \lVert\mathbf{u}\rVert, ((\sin \lVert\mathbf{u}\rVert)/\lVert\mathbf{u}\rVert) \mathbf{u})$ where $\lVert \cdot \rVert$ denotes the 2-norm.
\item Logarithm map: $ \ln(a, \mathbf{u}) = \left(\ln \sqrt{a^2 + \lVert \mathbf{u} \rVert^2}, \frac{1}{\lVert \mathbf{u} \rVert} 
\arccos \left( \frac{a}{\sqrt{a^2 + \lVert \mathbf{u} \rVert^2}} \right) \mathbf{u} \right)$. 
\item Power map: $(a, \mathbf{u})^{f(t)} = \exp(f(t)\ln(a, \mathbf{u}))$.
\end{itemize}
}

\begin{table} %\footnotesize
\setlength\tabcolsep{0pt} % let LaTeX figure out the amount of intercolunn whitespace
\begin{tabular*}{\textwidth}{@{\extracolsep{\fill}}  *{2}{C}}
\toprule
$$
\mbox{PFE:} 
\begin{array}{l}
\p^{n+1} = \mathcal{P}(\p^n + h f(\p^n,t^n))
\end{array}
$$
& 
$$
\mbox{PRK2:} \left\{
\begin{array}{l}
\s_1 = f(\p^n,t^n) \, , \, \q_1=\p^n+ h \s_1 \\
\s_2 = f(\q_1,t^{n+1})  \\
\p^{n+1} = \mathcal{P}(\p^n + \frac{1}{2} h (\s_1+\s_2))
\end{array}
\right.
$$
\\ \midrule
$$
\mbox{PRK3:}\left\{
\begin{array}{l}
\s_1 = f(\p^n,t^n) \, , \, \q_1=\p^n+\frac{1}{2}h \s_1 \\
\s_2 = f(\q_1,t^{n+1/2}) \\
\q_2=\p^n+2h \s_2 - h \s_1 \\
\s_3 = f(\q_2,t^{n+1})  \\
\p^{n+1} = \mathcal{P}(\p^n + \frac{h}{6} ( \s_1+4\s_2+\s_3 ))
\end{array}
\right.
$$
&
$$
\mbox{PRK4:}\left\{
\begin{array}{l}
\s_1 = f(\p^n,t^n) \, , \, \q_1=\p^n+\frac{1}{2}h \s_1 \\
\s_2 = f(\q_1,t^{n+1/2}) \, , \, \q_2=\p^n+\frac{1}{2}h \s_2  \\
\s_3 = f(\q_2,t^{n+1/2}) \, , \, \q_3=\p^n+h \s_3  \\
\s_4 = f(\q_3,t^{n+1})  \\
\p^{n+1} = \mathcal{P}(\p^n + \frac{h}{6} ( \s_1+2\s_2+2\s_3+\s_4 ))
\end{array}
\right.
$$
\\ 
\midrule
$$
\mbox{PTVDRK2:}\left\{
\begin{array}{l}
\q_1 = \p^n+h f(\p^n,t^n) \\
\q_2 = \q_1+h f(\q_1,t^{n+1}) \\
\p^{n+1} = \mathcal{P}(\frac{1}{2} (\p^n+\q_2))
\end{array}
\right.
$$
&
$$
\mbox{PTVDRK2':} 
\left\{
\begin{array}{l}
\q_1 = \mathcal{P}(\p^n+h f(\p^n,t^n)) \\
\q_2 = \mathcal{P}(\q_1+h f(\q_1,t^{n+1})) \\
\p^{n+1} = \mathcal{P}(\frac{1}{2} (\p^n+\q_2))
\end{array}
\right.
$$
\\ \midrule
$$
\mbox{PTVDRK3:}\left\{
\begin{array}{l}
\q_1 = \p^n+h f(\p^n,t^n) \\
\q_2 = \q_1+h f(\q_1,t^{n+1}) \\
\q_3 = \frac{1}{4} (3\p^n+\q_2) \\
\q_4 = \q_3+h f(\q_3,t^{n+1/2}) \\
\p^{n+1} = \mathcal{P}(\frac{1}{3}( \p^n+2\q_4))
\end{array}
\right.
$$
&
\mbox{PTVDRK3':} 
\left\{
\begin{array}{l}
\q_1 = \mathcal{P}(\p^n+h f(\p^n,t^n)) \\
\q_2 = \mathcal{P}(\q_1+h f(\q_1,t^{n+1})) \\
\q_3 = \mathcal{P}(\frac{1}{4} (3\p^n+\q_2)) \\
\q_4 = \mathcal{P}(\q_3+h f(\q_3,t^{n+1/2}) )\\
\p^{n+1} = \mathcal{P}(\frac{1}{3}( \p^n+2\q_4))
\end{array} 
\right.
\\
\addlinespace
\bottomrule
\end{tabular*}
\caption{(Section \ref{SubSec:ProjectedTVDRK}) {Some existing projected RK methods that we will compare with our proposed STVDRK methods.}}
\label{Table:PRK}
\end{table}

\subsection{Projected RK Methods}
\label{SubSec:ProjectedTVDRK}

By incorporating the projection concept into RK methods, we derive the following set of projected RK methods that effectively generate a flow on the unit sphere. 
{Although these methods are well-known to the community of geometrical integrations, we summarize these methods in Table \ref{Table:PRK} for a comprehensive overview.} The first two rows encompass simple projected RK methods, including projected-Forward Euler (PFE), projected-RK2 (PRK2), projected-RK3 (PRK3), and projected-RK4 (PRK4). The third row showcases projected-TVDRK2 (PTVDRK2) and its variant PTVDRK2', where projection is applied to each individual stage. The final row demonstrates projected-TVDRK3 (PTVDRK3) and its counterpart PTVDRK3', where projection is also implemented at every individual stage. The PTVDRK2' and PTVDRK3' methods involve a projection step at all intermediate stages and are referred to as internal projection methods \cite{hai11}. 
%These methods are useful when the velocity field cannot naturally extend to the neighborhood of the manifold or when the differential equation exhibits different stability behavior outside the manifold. Even though internal projection methods seem intuitive, the accuracy of this internal projection approach can be compromised. In particular, we will demonstrate in an example later and give a proof in Appendix B that the internal projection scheme PTVDRK3' is only second-order accurate.
It is important to note that the PTVDRK2 scheme is equivalent to PRK2. However, upon imposing the projection step for all intermediate stages, PTVDRK2' emerges as a distinct scheme. Notably, this PTVDRK2' scheme differs from PTVDRK2, even when utilizing the straightforward velocity extension described by equation (\ref{Eqn:VelocityExtension}). Specifically, the averaging steps are not equivalent between them. The quantities $\mathcal{P}(\frac{1}{2} (\mathbf{p}^n+\mathbf{q}_2))$ and $\mathcal{P}(\frac{1}{2} (\mathbf{p}^n+\mathcal{P}(\mathbf{q}_2)))$ generally yield different results. 

\section{Our Proposed Numerical Integrators}
\label{Sec:STVDRK}

This section introduces our proposed SLERP-TVDRK methods, abbreviated as STVDRK. The core idea involves mimicking the conventional TVDRK method but with two key alterations. Firstly, we replace the basic forward Euler step with the exponential map. This section introduces the exact solution for \textit{constant} motion on the unit sphere. Secondly, we modify the linear combination step by incorporating the Spherical Linear Interpolation (SLERP) interpolation technique. Comprehensive details will be presented in Sections \ref{SubSec:SFE} and \ref{SubSec:STVDRK}. Subsequent sections will provide insights into the properties of these methods.

\subsection{Spherical Forward Euler (SFE)}
\label{SubSec:SFE}

Consider moving the data point $\p^n$ with a nonzero velocity $\s = f(\p^n,t^n)$ for the period $h$. The arrival location on the unit sphere has an explicit formula given by $\p^{n+1} = \mbox{exp}_{\p^{n}}(h\s)$ where $\mbox{exp}_{\p}: T_{\p}\mathbb{S}^2 \rightarrow \mathbb{S}^2$ is the exponential map with $\mbox{exp}_{\p}(\s)=\gamma(1)$ where $\gamma$ being the unique geodesic satisfying $\gamma(0)=\p$ and $\gamma'(0)=\s$. Mathematically, we have the expression explicitly for the unit sphere $\mbox{exp}_{\p}(\s) = \cos(\|\s\|) \p + \sin(\|\s\|) \frac{\s}{\|\s\|}$, and therefore, we arrive the spherical forward Euler (SFE) method
$$
%\boxed{
\mbox{Spherical FE (SFE): } \, \left\{
\begin{array}{l}
\s_1 = f(\p^n,t^n) \\
\p^{n+1} = \cos(h \|\s_1\|) \p^n + \sin(h\|\s_1\|) \frac{\s_1}{\|\s_1\|} \, .
\end{array}
\right.
%}
$$
The SFE scheme emulates the conventional forward Euler method for solving ODEs in Cartesian space. However, in contrast to permitting the solution to reach any arbitrary point in the entire space after a single timestep, the SFE method respects the spherical geometry by ensuring that the solution remains confined to $\mathbb{S}^2$. We will provide a proof in Theorem \ref{Thm:SFE} to demonstrate that this scheme possesses local second-order accuracy, resulting in a globally first-order accurate solution.

\subsection{SLERP-TVDRK (STVDRK)}
\label{SubSec:STVDRK}
{Building upon the foundation of the SFE scheme, we propose the following STVDRK2 and STVDRK3 methods for solving ODEs on the unit sphere,}
$$
{
\mbox{STVDRK2:} \,
\left\{
\begin{array}{l}
\q_1 = \mbox{exp}_{\p^n}(hf(\p^n,t^n)) \\
\q_2 = \mbox{exp}_{\q_1}(hf(\q_1,t^{n+1})) \\
\p^{n+1} = \mbox{SLERP} ( \p^n,\q_2,0.5)
\end{array}
\right.
\quad \mbox{ and } \quad
\mbox{STVDRK3:} \,
\left\{
\begin{array}{l}
\q_1 = \mbox{exp}_{\p^n}(hf(\p^n,t^n)) \\
\q_2 = \mbox{exp}_{\q_1}(hf(\q_1,t^{n+1})) \\
\q_3 = \mbox{SLERP}( \p^n,\q_2,0.25) \\
\q_4 = \mbox{exp}_{\q_3}(hf(\q_3,t^{n+1/2})) \\
 \p^{n+1} = \mbox{SLERP}( \p^n,\q_4,2/3)
\end{array}
\right.
}
$$
The STVDRK2 method emulates the conventional TVDRK2 approach, {or a Trapezoidal method}, which involves two applications of the exponential map and the construction of a single SLERP operation. The initial step consists in substituting the first two forward Euler (FE) steps with SFE approximations. The pivotal aspect of the TVDRK2 method lies in the averaging step that yields the updated solution at the subsequent time level. However, directly applying the averaging concept to spherical ODEs could result in an approximation that violates the $\mathbb{S}^2$-constraint. Nevertheless, the averaging step can be interpreted as computing the midpoint along the geodesic formed by a linear interpolation between the initial and intermediate values. Thus, we replace the mean calculation by employing the SLERP technique on the unit sphere. This approach yields the midpoint along the geodesic path. {Similarly, the STVDRK3 method requires substituting all the forward Euler (FE) steps in the conventional TVDRK3 approach with SFE approximations and replacing any convex combination that arises between intermediate solutions with a SLERP incorporating a suitable parameter.}

To give an intuition of why the STVDRK2 scheme is second-order accurate, we consider a simple setup where the motion remains confined to the great circle of the unit sphere. The STVDRK2 scheme employs intermediate stages that lie on the unit sphere. For numerical representation, we introduce the angle representation. Consider the velocity given by $\theta'=\theta$, leading to the exact solution $\theta(h)=\theta_0 \exp(h) = (1+h+\frac{1}{2}h^2+\frac{1}{6}h^3+\cdots) \theta_0$. The SFE step introduces a polar angle change denoted by $\Delta \theta = h\theta^n$, resulting in $\theta^{n+1} = \left(1 + h \right) \theta^n$, indicating a first-order scheme. Following the second STVDRK stage, the solution becomes $\theta_2 = \left(1 + 2h + h^2\right) \theta^n$. The averaging step yields the final solution $\theta^{n+1} = \left(1 + h + \frac{1}{2} h^2\right) \theta^n$, achieving second-order accuracy. Translating back to the coordinates $\p^{n+1}$, the averaging step of the two intermediate stages is equivalent to SLERP interpolation.

\subsection{STVDRK-$r$ for $r\ge4$}

Unfortunately, the extension to $r$-th order methods with $r\ge4$ is not as straightforward. The difficulty lies in that many intermediate stages may be required to construct the TVDRK methods. Canceling higher-order terms in the Taylor expansion may not directly carry over from Cartesian space to the unit sphere. We have explored various $s$-stage order-4 explicit TVDRK methods (also known as strong-stability preserving (SSP) methods) but have only observed numerically up to third-order convergence. This section summarizes and reports these numerical schemes. Further careful investigations will be necessary, and they will be an important focus of future work.

In a more general context, this approach can be extended to any $s$-stage explicit TVDRK method for scalar ODEs \cite{shu88,shuosh88,gotshu98,gotshutad00} expressed in the form:
$$
\left\{
\begin{array}{l}
u^{(i)} = \sum_{k=0}^{i-1} \left[ \alpha_{ik} u^{(k)} + \beta_{ik} h f(u^{(k)}) \right] \, , \, i=1,2,\cdots,s, \\
u^{(0)} = u^n \mbox{ and } u^{n+1}=u^{(s)}
\end{array}
\right.
$$
where $\alpha_{ik}\ge 0$ with $\alpha_{ik}=0$ only if $\beta_{ik}=0$. This method adheres to the consistency condition $\sum_{k=0}^{i-1} \alpha_{ik}=1$ for $i=1,2,\cdots,s$. Specifically, when $\alpha_{ik},\beta_{ik}>0$, the $s$-stage method can be represented as a combination of forward Euler steps with varying time steps determined by $\beta_{ik}h/\alpha_{ik}$. Consequently, in line with the STVDRK scheme, we extend the update from:
$$
u^{(k)} + \frac{\beta_{ik} h }{\alpha_{ik}} f(u^{(k)}) \, \mbox{ to } \,
\q_{ik}=\mbox{exp}_{\q^{(k)}}\left[ \frac{\beta_{ik}h }{\alpha_{ik}} f(\q^{(k)}) \right]
$$ 
for $k=0,1,\cdots,i-1$ and $i=1,2,\cdots,s$, where $\q^{(k)}$ represents solutions at intermediate stages. However, the linear combination of $k$ from 0 to $i-1$ for more than two terms is not as straightforward. We have explored several approaches but have not been able to achieve satisfactory results. Obtaining high-order STVDRK-$r$ methods for $r\ge4$ will be an important focus of future work. We outline some approaches that we have tested.

One simple way to perform the convex combination is to apply SLERP to pairs of points progressively. For example, assuming that none of $\alpha_{i0}$ equals to 0, this approach leads to the following algorithm,
\begin{eqnarray*}
&& \q_{0}=\p^n \\
&& \left\{
\begin{array}{l}
\q_{ik}=\mbox{exp}_{\q_{k}}\left[ \frac{\beta_{ik}h }{\alpha_{ik}} f(\q_{k}) \right] \, \mbox{ for $k=0,1,\cdots,i-1$} \\
\br_{i1} = \mbox{SLERP}\left(\q_{i0},\q_{i1},\frac{\alpha_{i1}}{\alpha_{i0}+\alpha_{i1}}\right) \\
\br_{i2} = \mbox{SLERP}\left(\br_{i1},\q_{i2},\frac{\alpha_{i2}}{\alpha_{i0}+\alpha_{i1}+\alpha_{i2}}\right) \\
\quad\quad \vdots\\
\q_{i} = \br_{i,i-1} = \mbox{SLERP}\left(\br_{i,i-2},\q_{i,i-1},\frac{\alpha_{i,i-1}}{\alpha_{i0}+\cdots+\alpha_{i,i-1}}\right) \, \mbox{ for $i=1,\cdots,s$}
\end{array}
\right. \\
& & \p^{n+1}=\q_{s}\, .
\end{eqnarray*}
When a coefficient $\alpha_{ik}=0$, we can skip the corresponding SLERP interpolation step, as there is no contribution from that stage to the intermediate solution. Additionally, note that in the final intermediate step, we can utilize the consistency constraint to rewrite the expression as follows:
$$
\mbox{SLERP}\left(\br_{i,i-2},\q_{i,i-1},\frac{\alpha_{i,i-1}}{\alpha_{i0}+\cdots+\alpha_{i,i-1}}\right) =
\mbox{SLERP}\left(\br_{i,i-2},\q_{i,i-1},\alpha_{i,i-1}\right) \, .
$$
%We can also consider the following scheme as suggested in \cite{gotshu98}, 
%$$
%\left\{
%\begin{array}{l}
%u^{(1)} = u^n+0.500000000000000 h f(u^n) \\
%u^{(2)} = 0.405625000000000 \left[ u^n - 1.065687335761845 h f(u^n) \right] \\
%\quad\quad + 0.594375000000000 \left[ u^{(1)}+1.068486941019387 h f(u^{(1)}) \right] \\
%u^{(3)} = 0.021595600000000 \left[ u^n-0.947054029524533 h f(u^n) \right] \\
%\quad\quad + 0.240310650000000 \left[ u^{(1)}-1.065495848810696 h f(u^{(1)}) \right] \\ 
%\quad\quad + 0.738093750000000 \left[ u^{(2)}+1.066666666666667 h f(u^{(2)}) \right] \\
%u^{n+1} =  0.200000000000000 \left[ u^n+0.500000000000000 h f(u^n) \right] \\
%\quad\quad + 0.204233333333333 \left[ u^{(1)}+0.816060062020566 h f(u^{(1)}) \right] \\ 
%\quad\quad + 0.262433333333333 u^{(2)} \\
%\quad\quad + 0.333333333333333 \left[ u^{(3)}+0.500000000000000 h f(u^{(3)}) \right] \\
%\end{array}
%\right.
%$$
%In the STVDRK framework, we have the following scheme, named STVDRK4,
For example, based on the method as suggested in \cite{gotshu98}, we have the following STVDRK4 scheme,
$$%\footnotesize
%\boxed{
\left\{
\begin{array}{l}
\q_1 = \mbox{exp}_{\p^n}[0.500000000000000hf(\p^n)] \\ 
\q_{20} = \mbox{exp}_{\p^n}\left[- 1.065687335761845hf(\p^n)\right] \, , \, \q_{21} = \mbox{exp}_{\q_1}\left[1.068486941019387hf(\q^1)\right] \\
\q_2 = \mbox{SLERP}\left[ \q_{20},\q_{21},0.594375000000000\right] \\
\q_{30} = \mbox{exp}_{\p^n}\left[-0.947054029524533hf(\p^n)\right] \, ,\, \q_{31} = \mbox{exp}_{\q_1}\left[-1.065495848810696hf(\q^1)\right] \\
\q_{32} = \mbox{exp}_{\q_2}\left[1.066666666666667hf(\q_2)\right] \, , \, \br_{31} = \mbox{SLERP}\left[ \q_{30},\q_{31},0.917544541224197\right] \\
\q_3 = \mbox{SLERP}\left[ \br_{31},\q_{32},0.738093750000000\right] \\
\q_{40} = \mbox{exp}_{\p^n}\left[0.500000000000000hf(\p^n)\right] \, , \, \q_{41} = \mbox{exp}_{\q_1}\left[0.816060062020566hf(\q^1)\right] \\
\q_{43} = \mbox{exp}_{\q_3}\left[0.500000000000000hf(\q_3)\right] \, ,\, \br_{41} = \mbox{SLERP}\left[ \q_{40},\q_{41},0.505236249690773\right] \\
\br_{42} = \mbox{SLERP}\left[ \br_{41},\q_2,0.393650000000000\right] \\
\p^{n+1} = \mbox{SLERP}\left[ \br_{42},\q_{43},0.333333333333333\right]
\end{array}
\right.
%}
$$
Although we have some negative coefficients, as in the conventional TVDRK method, we can follow the same implementation to determine the exponential map. These negative $\beta_{i,j}$'s do not impose any computational challenges to the algorithm. The major drawback, however, is the efficiency of the overall method. It involves ten calls of the exponential maps and six evaluations of SLERP. The operation count is over three times of the STVDRK3. 

The main issue of this approach is that the progressive SLERP operator is \textbf{not} associative \cite{browor92,busfil01}. Furthermore, it has been proven in \cite{browor92} that no intrinsic definition of spherical averages satisfies associativity. Specifically, we take the above STVDRK4 scheme as an example. When constructing 
$\q_3$, it is possible to reverse the order and perform two SLERP operations as follows:
$$
\begin{array}{l}
\br_{31}' = \mbox{SLERP}\left[ \q_{32},\q_{31},0.245614850055866\right] \\
\q_3' = \mbox{SLERP}\left[ \br_{31}',\q_{30},0.021595600000000\right] \, .
\end{array}
$$
However, we have $\q_3\ne\q_3'$ in general, and these two implementations will produce two distinct numerical results.

The class of RK methods we introduce here are sometimes referred to as $s$-stage order-$s$ strong-stability preserving Runge-Kutta (SSPRK) methods. Alternatively, one might follow the approach outlined in \cite{spiruu02} to develop $s$-stage order-$p$ SSPRK methods with $s>p$ for solving ODEs on the unit sphere. In particular, 
%we rewrite the $5$-stages order-$4$ SSPRK method, denoted by SSPRK(5,4), in the form of \cite{shuosh88} as follows:
%$$
%\left\{
%\begin{array}{l}
%u^{(1)} = u^n + 0.39175222700392 h f(u^n) \\
%u^{(2)} =  0.44437049406734 u^n + 0.55562950593266 \left[ u^{(1)} + 0.663050807590193 h f(u^{(1)}) \right] \\
%u^{(3)} =  0.62010185138540 u^n + 0.37989814861460 \left[ u^{(2)} + 0.663050807607172 h f(u^{(2)}) \right] \\ 
%u^{(4)} = 0.17807995410773 u^n + 0.82192004589227 \left[ u^{(3)} + 0.663050807601060 h f(u^{(3)}) \right] \\ 
%u^{n+1} = 0.00683325884039 u^n + 0.51723167208978 u^{(2)} \\
%\quad\quad + 0.12759831133288 \left[ u^{(3)} + 0.663050807634935 h f(u^{(3)}) \right] \\
%\quad\quad + 0.34833675773694 \left[ u^{(4)} + 0.648818932180072 h f(u^{(4)}) \right]
%\end{array}
%\right.
%$$
%In the STVDRK framework, we have the following scheme, named SSSPRK(5,4),
we develop based on the $5$-stages order-$4$ SSPRK method, denoted by SSPRK(5,4) , and obtain the following SSSPRK(5,4) scheme,
$$%\footnotesize
%\boxed{
\left\{
\begin{array}{l}
\q_1 = \mbox{exp}_{\p^n}[0.39175222700392hf(\p^n)] \\ 
\q_{21} = \mbox{exp}_{\q_1}\left[0.663050807590193hf(\q_1)\right] \, , \, \q_2 = \mbox{SLERP}\left[ \p^n,\q_{21},0.55562950593266\right] \\
\q_{32} = \mbox{exp}_{\q_2}\left[0.663050807607172hf(\q_2)\right] \, , \, \q_3 = \mbox{SLERP}\left[ \p^n,\q_{32},0.37989814861460\right] \\
\q_{43} = \mbox{exp}_{\q_3}\left[0.663050807601060hf(\q_3)\right] \, , \, \q_4 = \mbox{SLERP}\left[ \p^n,\q_{43},0.82192004589227\right] \\
\q_{53} = \mbox{exp}_{\q_3}\left[0.663050807634935hf(\q_3)\right] \, , \, \q_{54} = \mbox{exp}_{\q_4}\left[0.648818932180072hf(\q_4)\right] \\
\br_{52} = \mbox{SLERP}\left[ \p^n,\q_2, 0.986961045402787\right] \, , \, \br_{53} = \mbox{SLERP}\left[ \br_{52},\q_{53},0.195804064212316\right] \\
\p^{n+1} = \mbox{SLERP}\left[ \br_{53},\q_{54},0.348336757736944\right]
\end{array}
\right.
%}
$$
This scheme requires six calls of the exponential map and six calls of SLERP interpolations. As discussed above, however, the current approach of implementing the convex combination creates an ambiguity due to the non-associative property in the SLERP operations. There is another challenge in this numerical scheme. Due to the loss in numerical precision in rewriting the scheme in the form of \cite{shuosh88}, we have observed that it might not be possible to maintain double precision accuracy in the computations. 
%Specifically, we suspect that various $\beta$ values might be identical, but they exhibit differences on the order of $O(10^{-11})$. 
Also, we will demonstrate in the numerical examples that the computed solutions do not exhibit an error of machine epsilon but rather on the order of $O(10^{-10})$ as we refine the timestep.

We have also considered the SSPRK(10,4) scheme as developed in \cite{ket08}:
$$%\small
\left\{
\begin{array}{l}
u^{(1)} = u^n \, , \, u^{(i+1)} = u^{(i)} +\frac{h}{6} f(u^{(i)}) \, \mbox{ for $i=1,2,3,4$} \\
u^{(6)} = \frac{3}{5} u^n +\frac{2}{5} \left[ u^{(5)}+ \frac{h}{6} f(u^{(5)}) \right] \, , \, u^{(i+1)} = u^{(i)} +\frac{h}{6} f(u^{(i)})  \, \mbox{ for $i=6,7,8,9$} \\
u^{n+1} = \frac{1}{25} u^n + \frac{9}{25} \left[ u^{(5)}+ \frac{h}{6} f(u^{(5)}) \right] + \frac{3}{5} \left[ u^{(10)}+ \frac{h}{6} f(u^{(10)}) \right] 
\end{array}
\right.
$$
and have investigated the following 10-stages order-4 method, SSSPRK(10,4) scheme:
$$%\footnotesize
%\boxed{
\left\{
\begin{array}{l}
\q_1 = \p^n \, , \, \q_{i+1} = \mbox{exp}_{\q_i}\left[\frac{h}{6} f(\q_i)\right] \, \mbox{ for $i=1,2,3,4$} \\
\q_{65} = \mbox{exp}_{\q_5}\left[\frac{h}{6} f(\q_5)\right] \, , \q_6 = \mbox{SLERP}\left[ \p^n,\q_{65},0.4 \right] \, , \,
\q_{i+1} = \mbox{exp}_{\q_i}\left[\frac{h}{6} f(\q_i)\right] \, \mbox{ for $i=6,7,8,9$} \\
\q_{10,10} = \mbox{exp}_{\q_{10}}\left[\frac{h}{6} f(\q_{10})\right] \, ,
\br_{10,5} = \mbox{SLERP}\left[ \p^n,\q_{65},0.9\right] \, , \, \p^{n+1} = \mbox{SLERP}\left[ \br_{10,5},\q_{10,10},0.6\right]
\end{array}
\right.
%}
$$
%Once again, we see two SLERP evaluations at the end of each time step.

An alternative approach to replace the convex combination is to determine the Fr\'echet mean \cite{busfil01,afstrovid13} of points on the unit sphere. This involves finding the minimizer of the Fr\'echet function, given by $\q=\argmin_{\p \in \mathbb{S}^2} \sum_{i=1}^n w_i \, \left[\mbox{dist}(\p,\p_i)\right]^2$ where $w_i$ are the weighting such that $\sum_{i=1}^n w_i=1$, and the function $\mbox{dist}(\p,\p_i)$ is the geodesic distance between two points $\q$ and $\p_i$ on the unit sphere. For example, $\q_3$ in the STVDRK4 can be determined by
$$\scriptsize
\argmin_{\p \in \mathbb{S}^2} \left[ 
0.0215956 \, \mbox{dist}(\p,\q_{30})^2 + 
0.24031065 \, \mbox{dist}(\p,\q_{31})^2  + 
0.73809375 \, \mbox{dist}(\p,\q_{32})^2 \right] \, .
$$
Unlike the corresponding problem in Cartesian space, this optimization problem generally does not have a closed-form solution. Therefore, the usual way to determine the Fr\'echet mean is through iterative methods. Although there are gradient descent approaches for solving this problem \cite{busfil01,afstrovid13,eichotwie19}, these schemes can make the overall algorithm for solving the ODE problem very inefficient. A more direct approach is to apply the same averaging as in conventional Cartesian space and then project the result back onto the unit sphere. This approach corresponds to using the initial condition in the iterative approach for the Fr\'echet mean. This can significantly accelerate the numerical computations. However, we have observed that neither approach yields a numerical solution with the expected accuracy. %Therefore, we have decided not to pursue further study in this direction.

\subsection{{Accuracy}}
\label{SubSec:Accuracy}

Since the exact solution stays on the unit sphere for all time, it will be essential to understand the change in the norm of the solution. This section considers the local error in the solution norm and tries to see how fast these numerical solutions converge to the sphere. Indeed, if we have an estimation of the convergence rate to the exact solution, we already have a corresponding estimate of the convergence rate to the sphere using the triangle inequality. But as we will see later, we can improve these estimations for some numerical schemes.

For simplicity, we consider autonomous flow. Let $\p^{\tiny \mbox{FE}} = \p^n+h \s_0$, $\p^{\tiny \mbox{SFE}} = \mbox{exp}_{\p^n}(h\s_0)$ and $\p^{\tiny \mbox{RK2}}$ be the \textbf{one-step} numerical solution obtained by the standard forward-Euler (FE) method (without projection), the spherical forward-Euler (SFE) method where $\s_0=f(\p^n)$ and the standard TVDRK2, respectively. Since there is no projection in determining $\p^{\tiny \mbox{FE}}$, the solution does not lie on $\mathbb{S}^2$. We further define $\p^{\tiny \mbox{PFE}}=\mathcal{P}(\p^{\tiny \mbox{FE}})$ as the projected-FE solution.

We have the following properties considering the norms of the approximations produced by different numerical integrators.
\begin{itemize}
\item $\left\|\p^{\tiny \mbox{FE}}\right\| = 1 +  O(h^2)$. Since $\p^n$ and $\s_0$ are orthogonal, we have $\|\p^{\tiny \mbox{FE}}\|=\sqrt{1+h^2}$ and the claim follows.
\item $\left\|\p^{\tiny \mbox{RK2}}\right\| = 1+ O(h^4)$. The proof of this statement is technical and is given in Appendix A. This result is surprising since we expect the RK2 solution should only have $O(h^3)$ local error bound. However, the accuracy in the solution norm has improved by one order.
\item $ \left\|\p^{\tiny \mbox{SFE}}\right\| =1 $. This is because the construction based on the exponential map stays on $\mathbb{S}^2$ by default.
\item Given that SLERP generates an interpolant on $\mathbb{S}^2$, solutions produced by the STVDRK method inherently stay on $\mathbb{S}^2$.
\end{itemize}

We also have the following properties concerning the accuracy of the numerical solutions. {We skip the proof of the proposition since it can be easily shown using Taylor's expansion and that $\p^n$ is orthogonal to $\s_0$.} Let $\p^*$ be the exact solution after one timestep.

\begin{prop}
$\left\| \p^{\tiny \mbox{FE}} - \p^{\tiny \mbox{SFE}} \right\| = O(h^2)$. 
\label{Prop:pFEpSFE}
\end{prop}

%\begin{proof}
%Applying the Taylor's expansion, we have 
%$$
%\p^{\tiny \mbox{FE}} - \p^{\tiny \mbox{SFE}} = \left( \frac{h^2}{2!}-\frac{h^4}{4!}+\cdots \right) \p^n + \left( \frac{h^3}{3!}-\frac{h^5}{5!}+\cdots \right) \s_0 \, .
%$$
%Since $\p^n$ is orthogonal to $\s_0$, the result follows.
%\end{proof}

\begin{thm}
$\|\p^{\tiny \mbox{SFE}} - \p^*\| =O(h^2)$.
\label{Thm:SFE}
\end{thm}
\begin{proof}
{
We have $\|\p^{\tiny \mbox{SFE}} - \p^*\| \le \|\p^{\tiny \mbox{SFE}} - \p^{\tiny \mbox{FE}} \| + \| \p^{\tiny \mbox{FE}} - \p^{\tiny \mbox{PFE}}\| + \| \p^{\tiny \mbox{PFE}} - \p^*\|$. We can bound the first term on the right-hand side by $O(h^2)$ using Proposition \ref{Prop:pFEpSFE}. 
%We also notice that the second term on the right-hand side is related to the norm of the solution $\p^{\tiny \mbox{FE}}$. 
Since $\left| \left\|\p^{\tiny \mbox{FE}}\right\| - 1 \right|$ is $O(h^2)$, we can conclude that $\| \p^{\tiny \mbox{FE}} - \p^{\tiny \mbox{PFE}}\|$ is $O(h^2)$. Moreover, the projected-FE method satisfies the local error bound $\|\p^{\tiny \mbox{PFE}} - \p^*\| =O(h^2)$ (see for example \cite{hai11}). Then, the statement follows.
%Therefore, since all three terms are bounded by $O(h^2)$, we can say that the SFE scheme has a second-order local error.
}
\end{proof}

\subsection{{Stability Conditions}}
\label{SubSec:Stability}

This section discusses the numerical constraints on the time step size in solving differential equations. 

In $\mathbb{R}^n$, a numerical integrator might require an A-stability condition that prevents oscillatory behavior in stiff systems. However, the discussion is not as straightforward for the ODE on $\mathbb{S}^2$. {We first consider an A-stability-like condition for our STVDRK schemes in this section.}

For example, if we mimic the case as in $\mathbb{R}^3$ and consider a linear system 
\begin{equation}
\mathbf{p}'(t)=M\mathbf{p}(t)
\label{Eqn:LinearModelProblem}
\end{equation}
with a 3-by-3 matrix $M$ on the sphere, we need to constrain the solution $\mathbf{p}(t)\in\mathbb{S}^2$ by requiring that $M$ be skew-symmetric, i.e., $M^T=-M$. Without loss of generality, we can express $M$ as
$$
M=\left(
\begin{array}{ccc}
0 & 0 & 0 \\
0 & 0 & \alpha \\
0 & -\alpha & 0
\end{array}
\right) \, ,
$$
where $\alpha$ is a constant in $\mathbb{R}$. This implies that the only possible motion is a rigid body rotation along the $x$-axis with an angular velocity of $|\alpha|$. In the general case, we can use an orthogonal change of coordinate system to express $M$ in a different form that rotates the $x$-axis to a different direction. However, the resulting motion remains a rigid body rotation on the unit sphere. Therefore, in contrast to the situation in $\mathbb{R}^3$, the linear ODE on $\mathbb{S}^2$ does not possess an attractive stable equilibrium for the typical stability analysis.

Instead of considering a skew-symmetric matrix, we will look at a general linear ODE $\mathbf{p}'(t)=M\mathbf{p}(t)$ in $\mathbb{R}^3$. We will project the solution onto $\mathbb{S}^2$. This gives us a \textit{nonlinear} ODE:
\begin{equation}
\mathbf{q}'=M\mathbf{q} - \mathbf{q}\mathbf{q}^T M \mathbf{q} = (I - \mathbf{q}\mathbf{q}^T ) M \mathbf{q} = g(\mathbf{q}) \, .
\label{Eqn:NonlinearModelProblem}
\end{equation}
Here, $\mathbf{q}(t)=\mathbf{p}(t)/\|\mathbf{p}(t)\| \in \mathbb{S}^2$, and the right-hand side of the equation represents the residual of $M\mathbf{q}$ after subtracting off the projection onto the direction $\mathbf{q}$. This guarantees that the velocity is perpendicular to $\mathbf{q}$, and so the evolution will stay on $\mathbb{S}^2$ when the initial condition $\mathbf{q}(0)\in\mathbb{S}^2$. Furthermore, for any eigenvector $\mathbf{v}\in\mathbb{S}^2$ of the matrix $M$, we have $g(\mathbf{v}) = M\mathbf{v} - \mathbf{v}\mathbf{v}^T M \mathbf{v} = \lambda\mathbf{v} - \lambda \mathbf{v}\mathbf{v}^T \mathbf{v} = \lambda\mathbf{v} - \lambda \mathbf{v} \cdot 1 = 0$. This implies that any eigenvector of $M$ will lead to an equilibrium of the dynamical system. Moreover, $-\mathbf{v}$ also gives an equilibrium of $g$ since $g(-\mathbf{v})$ is zero.

The Jacobian of the function $g$ is given by $Dg(\mathbf{q})=M-(\mathbf{q}^TM\mathbf{q}) I -2\mathbf{q}\mathbf{q}^TM$. To study the stability at the equilibrium $\mathbf{v}$, we need to look at the eigenvalue of the corresponding matrix $Dg(\mathbf{v})$. The simplest case assumes a constant diagonal matrix $ M $ with non-zero diagonal elements. This implies that the eigenvectors of $M$ are given by the direction vectors $\mathbf{e}_i$. Therefore, we have $Dg(\mathbf{e}_i) \mathbf{e}_j = M \mathbf{e}_j - (\mathbf{e}_i^T M \mathbf{e}_i) \mathbf{e}_j - 2\mathbf{e}_i\mathbf{e}_i^T M \mathbf{e}_j = (\lambda_j-\lambda_i) \mathbf{e}_j$ for $i\ne j$ and $Dg(\mathbf{e}_i) \mathbf{e}_i = -2\lambda_i \mathbf{e}_i$. Since we are considering motions on the surface, the motion direction $\mathbf{e}_i$ at the point $\mathbf{e}_i$ is actually not the most crucial factor in affecting the stability of the trajectory. Instead, we need to look at those other two directions. The stability of the motion is therefore governed by the quantities $\lambda_j-\lambda_i$. Defining $\sigma_{i,j}=\lambda_j-\lambda_i$ for $i\ne j$ and $\sigma=\min_{j,i\ne j} \sigma_{i,j}$, we have a stable equilibrium at $\mathbf{q}=\mathbf{e}_i$ if $\sigma_{i,j}<0$ for both $i\ne j$.

We now use the above property to mimic an A-stability condition in the numerical integrator. We assume $\sigma_{i,j}<0$ for both $i\ne j$ at the point $\mathbf{q}=\mathbf{v}_i$, we have the {following stability conditions} for the above nonlinear model problem (\ref{Eqn:NonlinearModelProblem}).
\begin{itemize}
\item We first examine methods based on forward Euler or RK2 schemes. This category encompasses the conventional forward Euler method, SFE, RK2, TVDRK2, PRK2, PTVDRK2, and STVDRK2. In this context, the stability condition remains consistent with numerous numerical integrators for standard ODEs, given by $-2 \leq \sigma h \leq 0$ {with $h$ representing the time stepsize in the numerical integrator.}
\item For methods related to RK3 (including RK3, PRK3, TVDRK3, PTVDRK3, and STVDRK3), we have the interval of absolute stability given by $\mu^* \le \sigma h \le 0$ where the lower bound $\mu^* \approx -2.513$ is given by the real root of the equation $\frac{1}{6} \mu^3 + \frac{1}{2} \mu^2 + \mu +2 =0$. 
\end{itemize}

In particular, we can consider
$$
M=\left(
\begin{array}{ccc}
\frac{1}{2} & 0 & 0 \\
0 & -\frac{1}{2} &0 \\
0 & 0 & -\frac{1}{2}
\end{array}
\right) \, ,
$$
At the equilibrium $\mathbf{q}=\mathbf{e}_1$ (and the opposite point $\mathbf{q}=-\mathbf{e}_1$), the two relevant eigenvalues of $Dg(\mathbf{e}_1)$ are both given by $-\frac{1}{2}-\frac{1}{2}=-1$, indicating that this point is an attractive equilibrium. Similar calculations reveal that the remaining two eigenvectors of matrix $M$, namely $\mathbf{e}_2$ and $\mathbf{e}_3$, correspond to unstable equilibria. Consequently, we deduce that $\sigma = -1$. When applying this analysis to diverse numerical integrators, the corresponding stability conditions emerge as follows.
\begin{itemize}
\item We first consider the forward Euler-based or RK2-based methods. This includes the standard forward Euler method, the SFE, RK2, TVDRK2, PRK2, PTVDRK2, and STVDRK2. We have the stability condition $h \le 2$. 
\item For methods related to RK3 (including RK3, PRK3, TVDRK3, PTVDRK3, and STVDRK3), we have the interval of absolute stability given by $h \le -\mu^* $.
\end{itemize}

\begin{figure}[!htb]
\centering
%(a)\includegraphics[width=0.35\textwidth]{figures/Stability1}
%(b)\includegraphics[width=0.35\textwidth]{figures/Stability2} 
\includegraphics[width=0.75\textwidth]{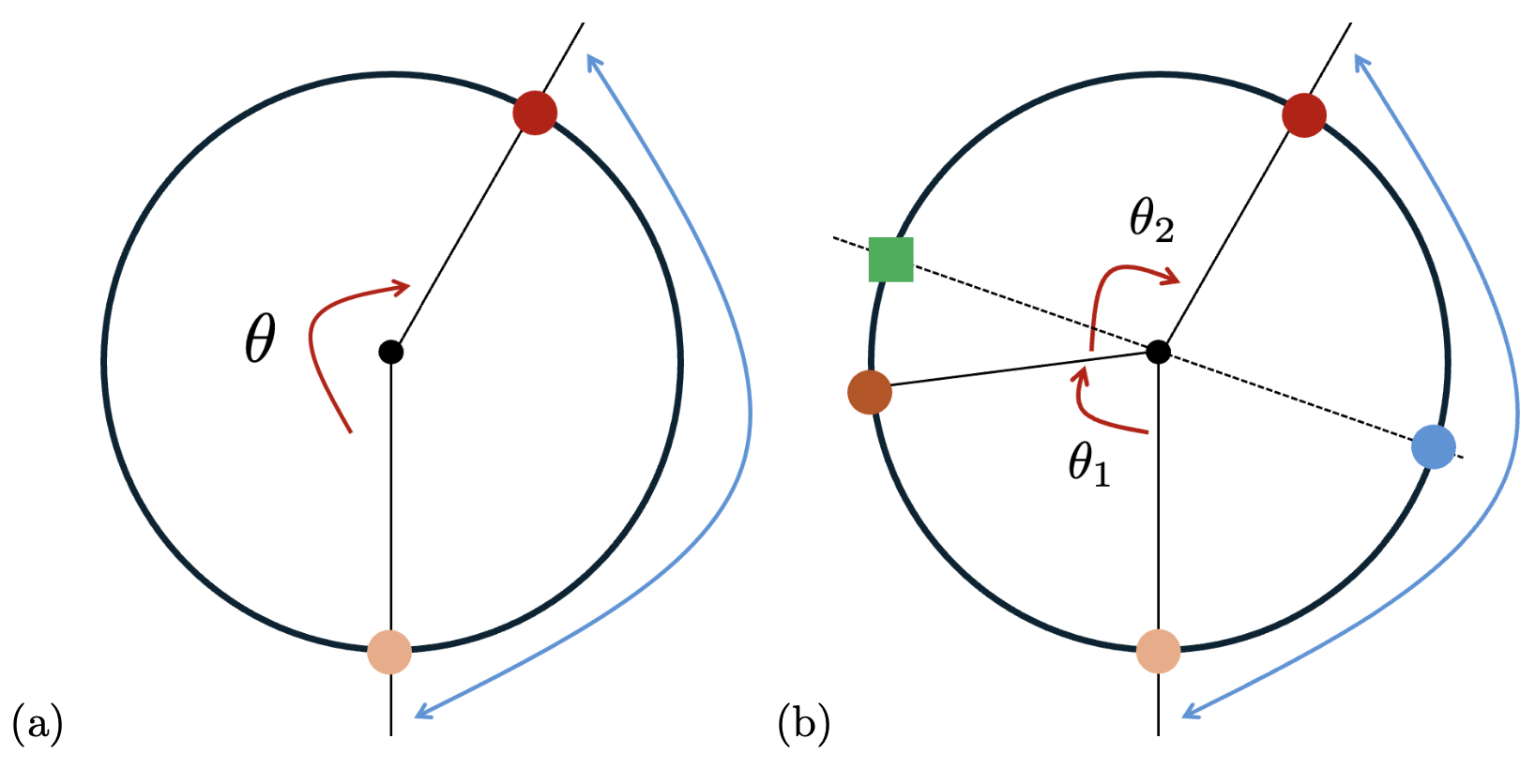}
\caption{{
(Section \ref{SubSec:Stability} {Stability Condition Due to SLERP}) We assume the motion is defined on the great circle of the sphere with the initial location on the south pole (i.e., the orange dot). (a) If the initial location travels more than $h\| f(\mathbf{p},t) \|>\pi$, the geodesic to the final location (the red dot) will be given by the segment indicated by the blue curve instead of the segment of travel. (b) For STVDRK2, we assume $\mathbf{p}^n$ (the orange dot) travels to $\mathbf{q}_1$ (the dark orange dot) and then $\mathbf{q}_2$ (the red dot). The midpoint along the geodesic from $\mathbf{p}^n$ to $\mathbf{q}_2$ is given by the blue dot instead of the expected location indicated by the green square.
}} 
\label{Fig:StabilitySLERP}
\end{figure}

{Except for a condition motivated by the A-stability as discussed above}, our class of STVDRK has another constraint on the step size due to the geometry. We begin by considering the SFE method. Let $\gamma$ be the curve that joins $\mathbf{p}^n$ and $\mathbf{p}^{n+1}$ based on the exponential map, where $\mathbf{p}^{n}=\gamma(0)$ and $\mathbf{p}^{n+1}=\gamma(hf(\p^n,t^n))$. We require that this curve $\gamma$ must be identical to the interpolant formed through SLERP. However, in scenarios where a point strays too far from its initial position, the SLERP interpolation could potentially yield the alternative segment along the great circle. To guarantee equivalence, satisfying the condition $h\|f(\p,t)\| < \pi$ becomes necessary, {as demonstrated in Figure \ref{Fig:StabilitySLERP}(a).}

Let us now consider the STVDRK2 method, {as shown in Figure \ref{Fig:StabilitySLERP}(b).} An initial observation reveals that for any selection of $h$ and $f$, the scheme generates three points $\mathbf{q}_1$ {(the dark orange dot)}, $\mathbf{q}_2$ {(the red dot)}, and $\mathbf{p}^n$ {(the orange dot)} that reside on the same hemisphere of the unit sphere. To see this, we choose a unit vector $\mathbf{n}$ that is perpendicular to the plane defined by the three points. Specifically, we set $\mathbf{n}=(\mathbf{q}_1-\mathbf{q}_2)\times (\mathbf{q}_1-\mathbf{p}^n)$. We then consider the plane that passes through the origin with the normal vector given by $\mathbf{n}$. We define the function $\phi(\mathbf{p})=\mathbf{n}\cdot \mathbf{p}$ for $\mathbf{p}\in\mathbb{R}^3$, and the equation of the plane is given by ${\phi^{-1}(0)}$. It can be shown that $\phi(\mathbf{q}_1)=\phi(\mathbf{q}_2)=\phi(\mathbf{p}^n)=(\mathbf{q}_2 \times \mathbf{p}^n)\cdot \mathbf{q}_1$ are all equal. Therefore, if the three points are not collinear (i.e., not on the same great circle), then they all lie on the same hemisphere separated by the plane $\phi(\mathbf{p})=0$. Now, we address the restriction on the step size $h$. When all intermediate stages satisfy the condition $h\|f(\p,t)\| < \pi/2$, two SLERP interpolations are established (one linking $\mathbf{p}^n$ and $\mathbf{q}_1$, and one joining $\mathbf{q}_1$ and $\mathbf{q}_2$) which both remain within the same hemisphere. Consequently, the midpoint between $\p^n$ and $\mathbf{q}_2$ similarly stays within this hemisphere. As a result, the STVDRK2 consists of consistent interpolations.

Extending this rationale to the STVDRK3 scheme, we require that $h\|f(\p,t)\| < \pi/2$ to ensure that the SLERP interpolant, responsible for generating $\q_3$, remains within the same hemisphere encompassing $\p^n$, $\q_1$, and $\q_2$. This condition alone does not seem enough to prevent non-uniqueness in the SLERP interpolation within the overall scheme. We might also need to additionally ensure that $\p^n$, $\q_3$, and $\q_4$ lie on a hemisphere (note that this hemisphere could potentially differ from the one encompassing $\p^n$, $\q_1$, and $\q_2$). However, since the distance between $\p^n$ and $\q_3$ cannot exceed $\pi/4$, the next propagation step cannot propel $\q_4$ further away from $\p^n$ by a distance greater than $\pi$. As a result, the scheme avoids the possibility of mis-interpolation through SLERP. Thus, the condition $h\|f(\p,t)\| < \pi/2$ itself is enough to ensure stability, given the constraints introduced by SLERP.

\subsection{{Comparisons with Some Lie-group Integrators}}

{In this section, we compare our proposed numerical integrators with some existing approaches. Since our approach is expressed in the form of exponential maps, it shares some similarities with Lie-group methods. We introduced one Lie-group approach in Section \ref{Sec:Introduction}, which is developed based on the inverse of the derivative of the exponential map, as shown in equation (\ref{Eqn:RKonManifolds}). Restricted to the ordinary differential equation (ODE) given by $Y’ = A(Y)Y$, where $Y \in G$ is in a matrix Lie group denoted by $G$, and $A(Y) \in \mathfrak{g} = T_I G$ is the corresponding Lie algebra, we obtain an approach that determines a matrix function $\Omega(t)$ such that $Y(t) = \exp(\Omega(t)) Y_0$. One can show that $\Omega(t)$ satisfies the differential equation given by $\Omega’(t) = d\exp_{\Omega}^{-1}(A(t))$, where the inverse of the derivative of the matrix exponential is given by $d\exp_{\Omega}^{-1}(H) = \sum_{k \ge 0} \frac{B_k}{k!} \text{ad}{\Omega}^k (H)$. Here, $B_k$ denotes the Bernoulli numbers, and $\text{ad}{\Omega}(A) = \Omega A - A \Omega$ is the adjoint operator. This leads to the so-called Magnus approach \cite{mag54}. Although these methods can achieve high-order accuracy, they heavily rely on certain approximations of the inverse mapping. In contrast, the class of STVDRK methods is purely explicit and does not involve the computation of any inverse mapping.}

{Another Lie-group method is the explicit $s$-stage Crouch-Grossman approach \cite{crogro93}, given by:
$$
Y^{(i)} = \exp(ha_{i,i-1}K_{i-1}) \cdot \dots \cdot \exp(h a_{i,1} K_1) Y_n \, , \,
Y_{n+1}=\exp(hb_sK_s) \cdot \dots \cdot \exp(h b_1 K_1) Y_n,
$$
where $K_i = A(Y^{(i)})$, and $b_i$ and $a_{i,j}$ are some real numbers for $i,j = 1,\cdots,s$. Since this class of methods involves only the composition of exponential maps, the numerical solution remains on the manifold by definition. These methods have a flavor similar to typical RK methods by successively correcting the manifold flow direction to achieve high-order accuracy. In contrast, our approach corrects the solution’s location by considering the geodesic on the manifold using SLERP interpolation. This class of methods requires many evaluations of the exponential map. Specifically, the $s$-stage method mentioned above might require as many as $\frac{1}{2}s(s+1)$ evaluations of the exponential map. In contrast, our proposed STVDRK methods apply to a large class of nonlinear ODE problems and are more economical (at least up to STVDRK3), as STVDRK2 and STVDRK3 consist of only two and three calls to the exponential map, respectively.}

{Taking the simplest case of the Magnus method or the Crouch-Grossman method, one obtains the so-called Lie-Euler method $Y_{n+1} = \exp(hA(Y_n))Y_n$. This method has a similar flavor to our SFE formula, except that our expression applies to general nonlinear dynamical systems on spherical geometry. Additionally, the approaches to high-order extensions of these methods are different. In particular, none of these Lie-group methods involve interpolation of intermediate solutions on the underlying manifold.}

\section{Numerical Examples}
\label{Sec:Examples}

This section compares our proposed STVDRK methods with several other numerical integrators. We aim to demonstrate the accuracy and stability of the proposed algorithms through numerical demonstrations. Furthermore, we employ these methods in diverse applications, including %Hamiltonian flow simulations and 
simulations of high-frequency wave propagations on a unit sphere utilizing ray tracing methods and exploring $p$-harmonic flows fundamentally governed by a partial differential equation (PDE). 

\subsection{Convergence}
\label{SubSec:ExConvergence}

This example considers the following four-point vertices flow given by
$$
f(\x)=\sum_{i=1}^4 \frac{\x_i \times \x}{2(1-\x_i\cdot \x)}
$$
with $\x_1=\frac{1}{\sqrt{3}}(1,-1,1)$, $\x_2=\frac{1}{\sqrt{3}}(1,-1,-1)$, $\x_3=\frac{1}{\sqrt{5}}(-2,1,0)$ and also $\x_4=\frac{1}{\sqrt{2}}(-1,-1,0)$. The initial condition is given by $\p_0=(1,0,0)$. {Except those explicit schemes mentioned in Section \ref{SubSec:ProjectedTVDRK}, we will also compare with the solutions obtained by various popular numerical integrators. The first is the fixed step third-order Radau-IIA method \cite{haiwan99,hailubwan06}, which is a fully implicit RK method. Although the Jacobian of the dynamical systems can be determined analytically, we approximate it using finite differences. We have also considered the high-order fixed-step DOP54 scheme \cite{haiwannor93} developed based on the explicit RK method of order (4)5 \cite{dorpri80}. Finally, we note that the differential equation can also be formulated as a differential-algebraic system of equations (DAE) of index-1 by replacing one differential equation (for example, the differential equation related to the $z$-component) using the derivative of the constraint given by $xx'+yy'+zz'=0$. We will compare our solutions with the third- and fifth-order solutions from the \textsf{MATLAB} function \textsf{ode15s} by specifying the parameter \textsf{MaxOrder} using 3 and 5, respectively. To ensure a fair comparison, we set both \textsf{RelTol} and \textsf{AbsTol} to be the third and fifth power of the timestep, respectively, and switch off the adaptivity in the numerical integrator by enforcing fixed timesteps in the algorithms. We label these integrators DAE3 and DAE5. Based on the DEA formulation, we have also compared our solutions with the \textsf{Octave} implementation of the DASSL method \cite{brecampet95}. Unfortunately, we have less control of this adaptive numerical integrator. We can only specify the absolute tolerance and the maximum order, but not a fixed timestep. In particular, we vary the bound in the absolute tolerance using the third and fourth power of a parameter corresponding to the actual fixed timestep in other numerical schemes and assign the maximum order to be three and four, respectively. We label these two integrators DASSL3 and DASSL4.}

\begin{table}[!htb]
\centering
\begin{tabular}{|c||c|c||c|c|c|}
\hline
                 & (a)  $E_2$           & (b) $E_{\mbox{norm}}$ & (c) Proj & (d) SLERP & (e) exp\\ \hline\hline
RK3              & 3                           & 3      & --   & --  & -- \\
RK4              & 4                           & 4    & --   & --  & --\\ 
TVDRK2           & 2                           & \underline{\textbf{3}}          & --            & --    &   --  \\
TVDRK3           & 3                           & 3                       & --   & --   & -- \\
\hline
PRK3   & 3                           & Exact            & 1          & --    &  --  \\
PRK4    & 4                           & Exact             &1        & --    &  --   \\ 
PTVDRK2 & 2                           & Exact               & 1        & --   &  --  \\
PTVDRK2' & 2                           & Exact               & 3        & --   &  --  \\
PTVDRK3 & 3 				   & Exact              & 1         & --  &   --  \\
PTVDRK3' & \underline{\textbf{2}} & Exact              &5         & --   &   -- \\
\hline
SFE         & 1                           & Exact                  &  --     &  --  &  1   \\
STVDRK2     & 2                           & Exact             & --       &  1     & 2   \\
STVDRK3     & 3                           & Exact              & --      &    2  &  3 \\ 
\hline
STVDRK4     & *                         & Exact              & --      &   6  & 10  \\ 
SSSPRK(5,4)     & *                        & Exact              & --      &   6  &  6 \\ 
SSSPRK(10,4)     & *                       & Exact              & --      &   3  &  10 \\
\hline
\end{tabular}
\caption{(Section \ref{SubSec:ExConvergence}) The order of convergence for various numerical integrators is demonstrated when the error is measured using (a) $E_2$ and (b) $E_{\mbox{norm}}$. Since the constraint $\| \mathbf{p}^n \| = 1$ is satisfied for all projected methods and integrators developed based on SLERP, we do not report the convergence rate for these methods. Five specific values of interest are highlighted: the accuracy of the PTVDRK3', STVDRK4, SSSPRK(5,4), and SSSPRK(10,4) schemes and the convergence rate in the norm of the TVDRK2 scheme. These numbers are either bolded or replaced by *. The number of (c) projections, (d) SLERP calls, and (e) exponential map evaluations are presented.
}
\label{Table:Convergence}
\end{table}

\begin{figure}[!htb]
\centering
(a)\includegraphics[trim=0 0 10 0, clip, width=0.45\textwidth]{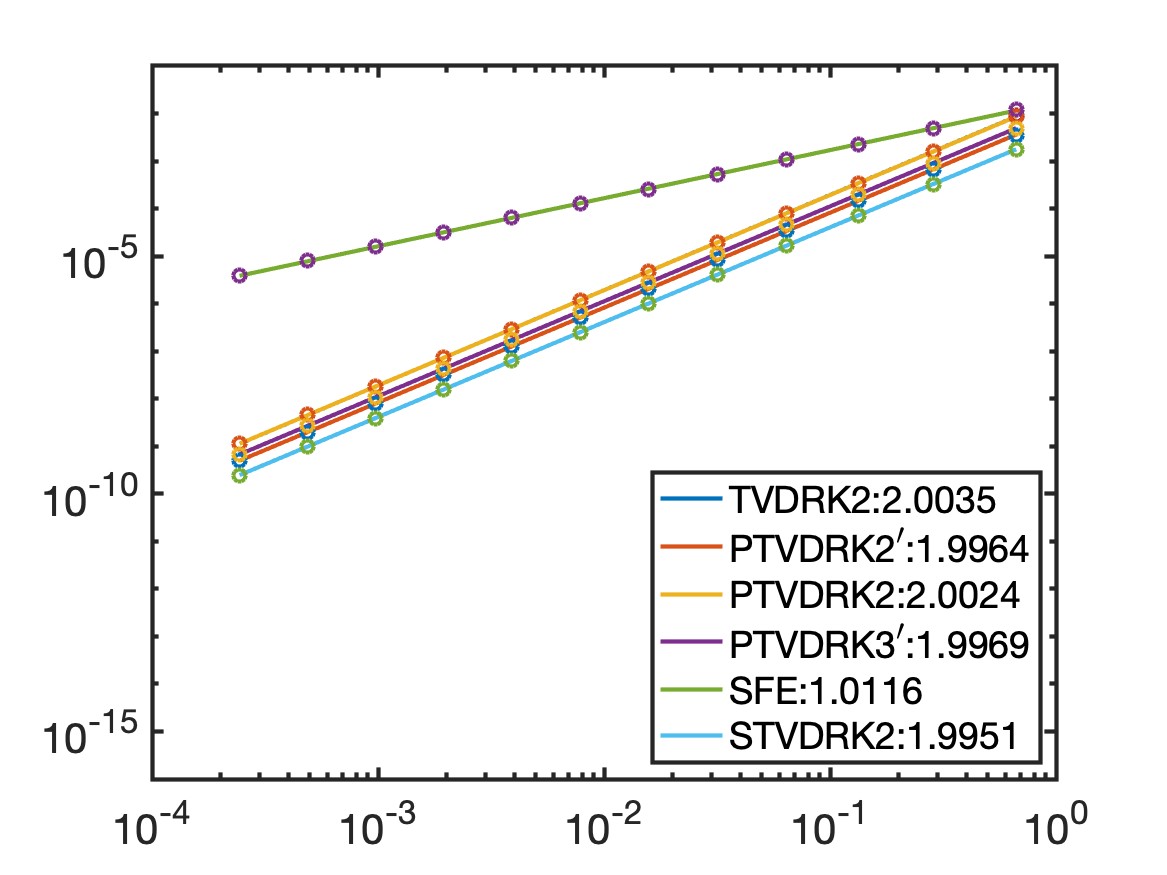}
(b)\includegraphics[trim=0 0 10 0, clip, width=0.45\textwidth]{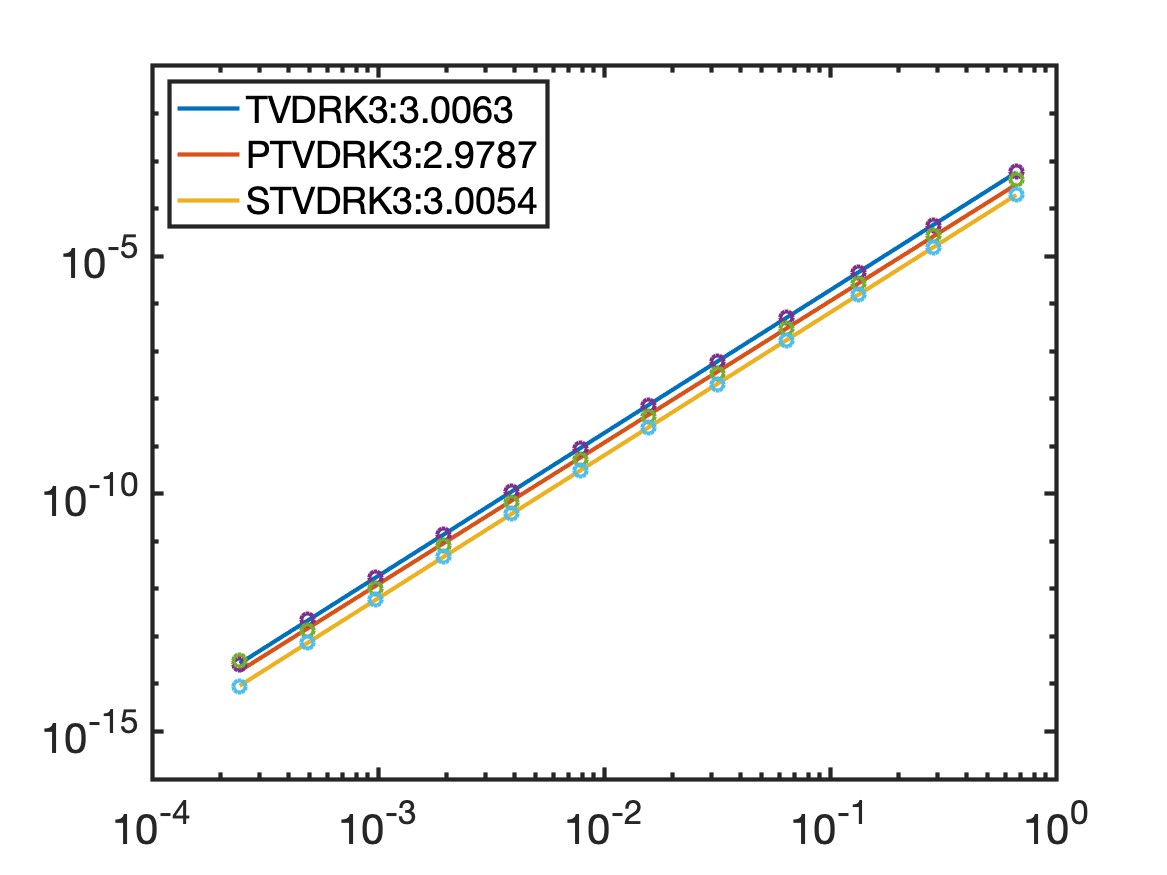} 
(c)\includegraphics[trim=0 0 10 0, clip, width=0.45\textwidth]{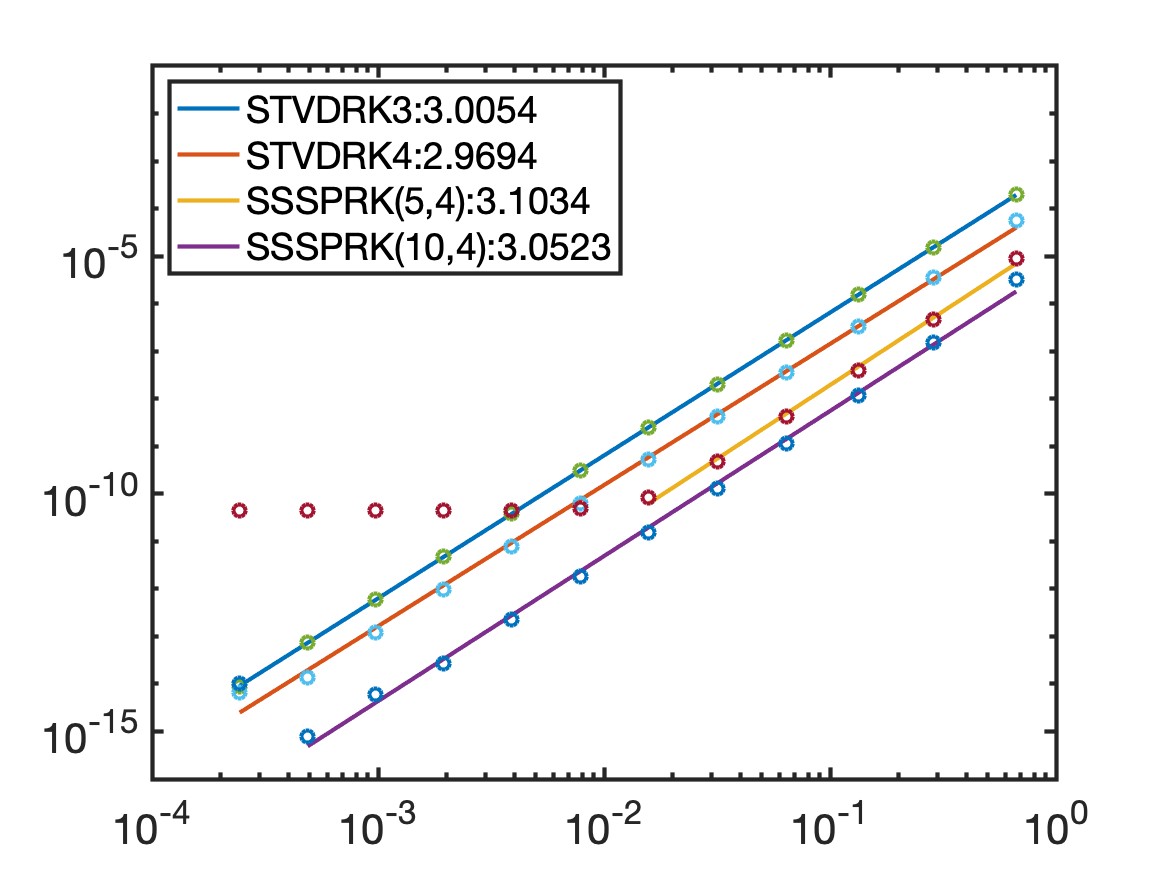} 
(d)\includegraphics[trim=0 0 10 0, clip, width=0.45\textwidth]{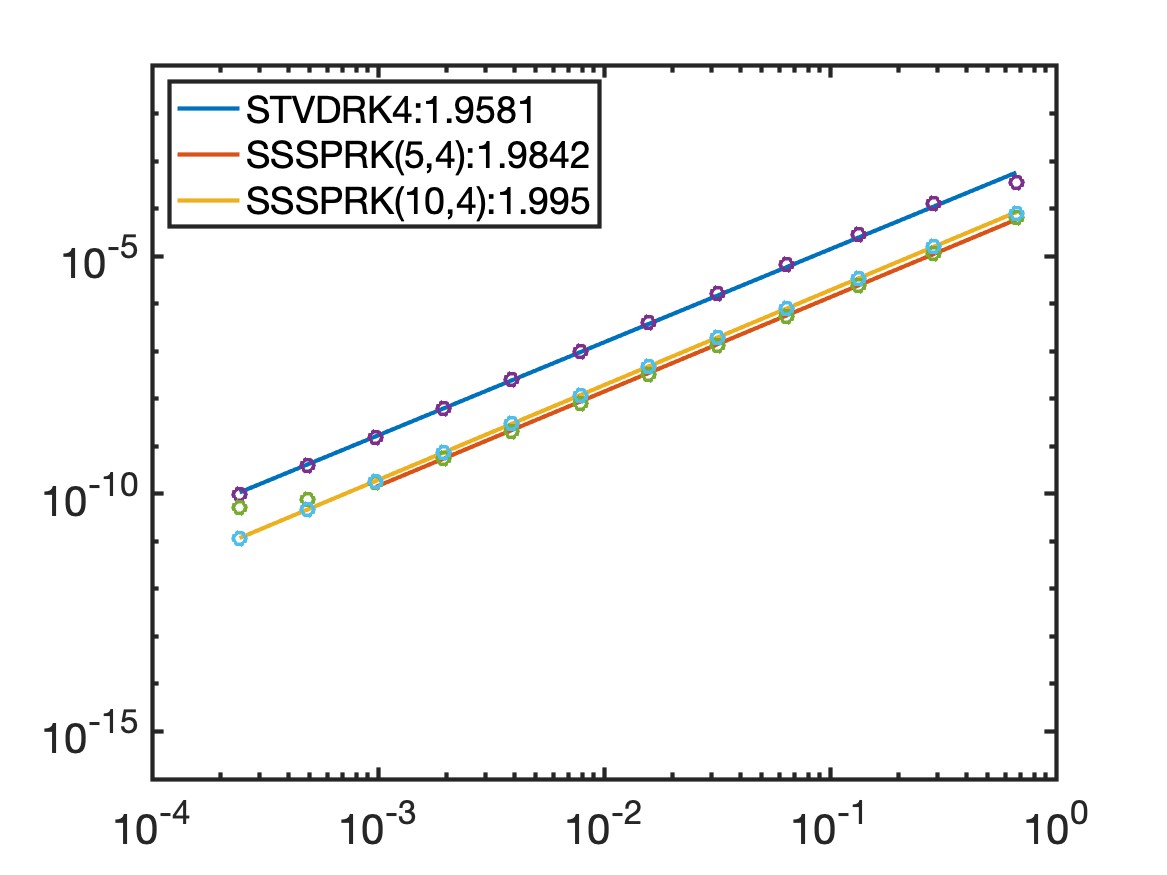}
{(e)}\includegraphics[trim=0 0 10 0, clip, width=0.45\textwidth]{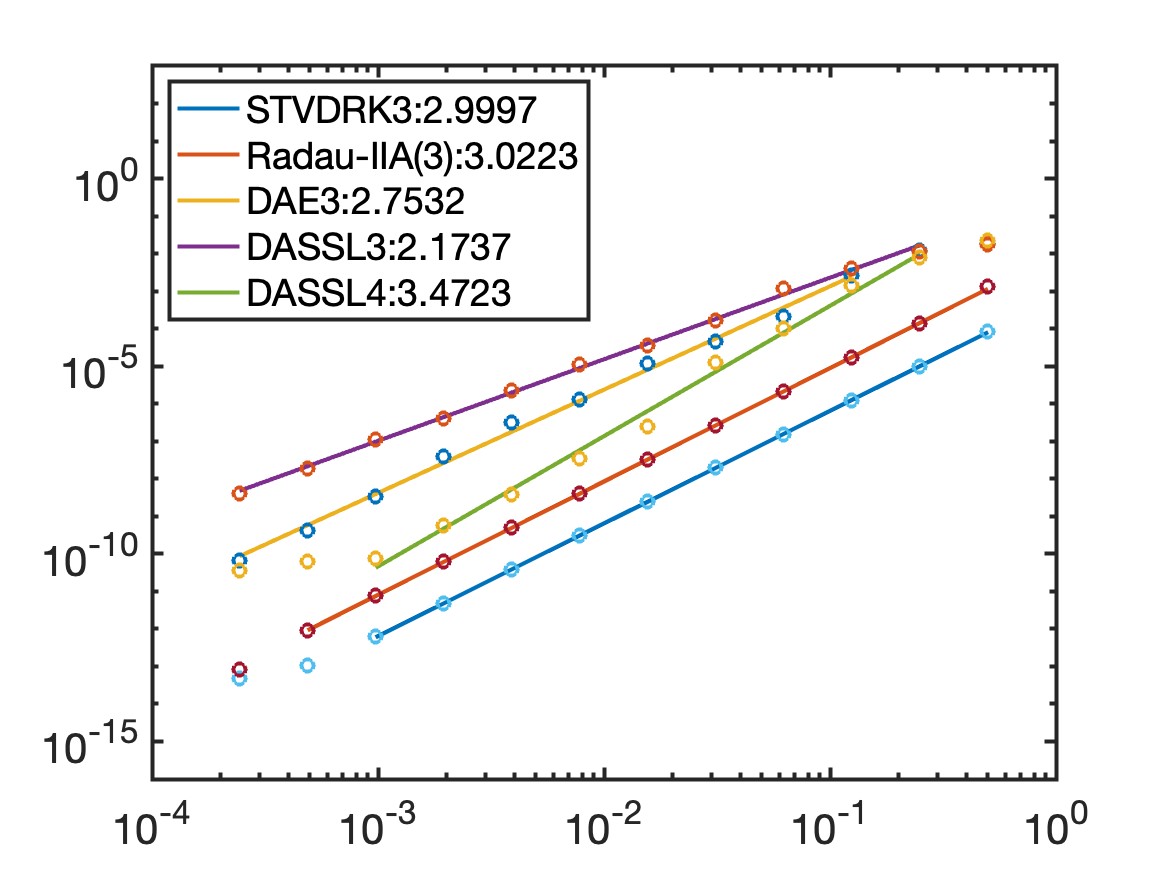} 
{(f)}\includegraphics[trim=0 0 10 0, clip, width=0.45\textwidth]{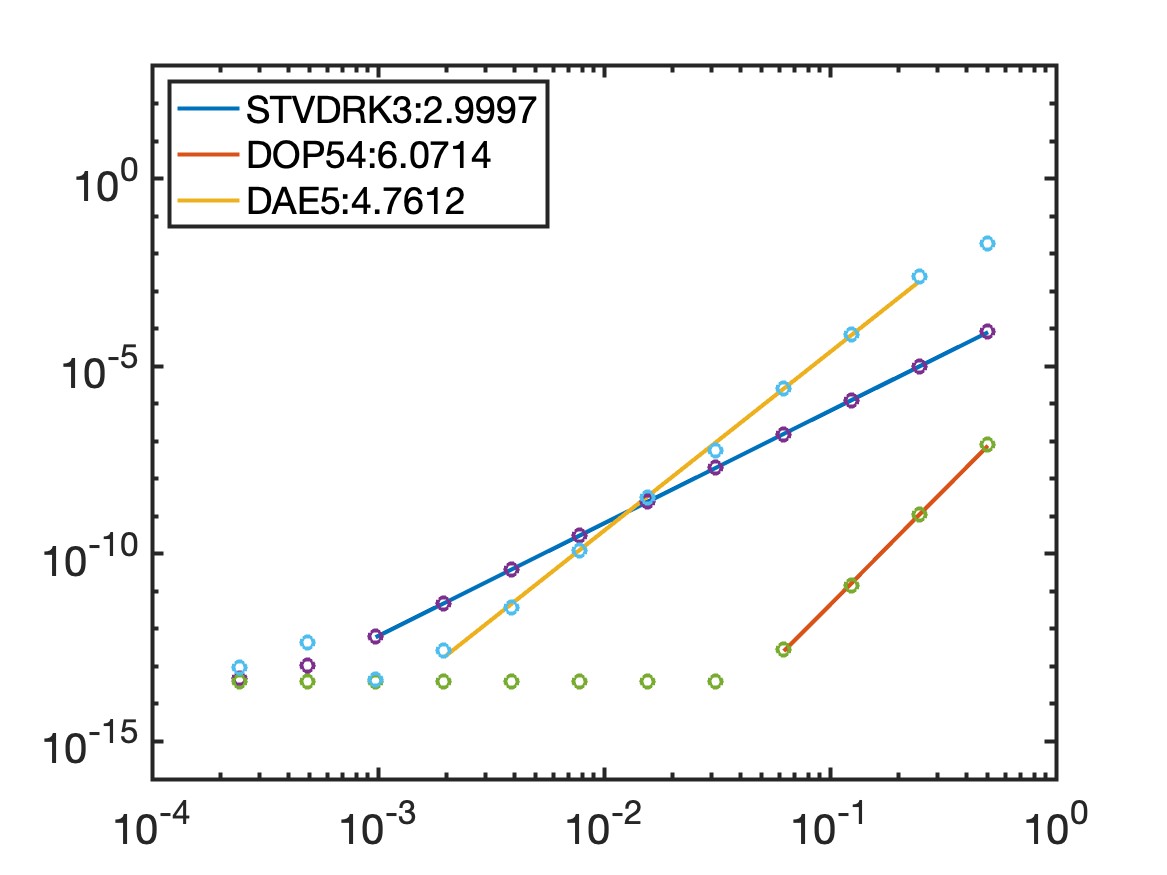} 
\caption{(Section \ref{SubSec:ExConvergence}) (a) The $E_2$ errors in the solution obtained by SFE demonstrate first-order accuracy, while those obtained by TVDRK2, PTVDRK2, PTVDRK2', PTVDRK3', and our proposed STVDRK2 demonstrate second-order accuracy. (b) The $E_2$ errors in the solutions obtained by TVDRK3, PTVDRK3, and our proposed STVDRK3 exhibit third-order convergence. (c) The $E_2$ errors in the solutions obtained by STVDRK4, SSSPRK(5,4), and SSSPRK(10,4) when the convex combination is done using progressive SLERP. We compare these errors with the one computed using STVDRK3. All schemes exhibit only third-order convergence. (d) The $E_2$ errors in the solutions obtained by STVDRK4, SSSPRK(5,4) and SSSPRK(10,4) when the convex combination is done using the Fr\'echet mean. All schemes exhibit only second-order convergence. {(e) The $E_2$ errors in the solutions obtained by Radau-IIA(3), DAE3 and DASSL methods. We have also shown the error in the STVDRK3 solutions as a reference. (f) The $E_2$ errors in the solutions obtained by high-order methods DOP54 and DAE5. We have also shown the error in the STVDRK3 solutions as a reference.}
} 
\label{Fig:Convergence123}
\end{figure}

\begin{figure}[!htb]
\centering
(a)\includegraphics[trim=0 0 10 0, clip, width=0.45\textwidth]{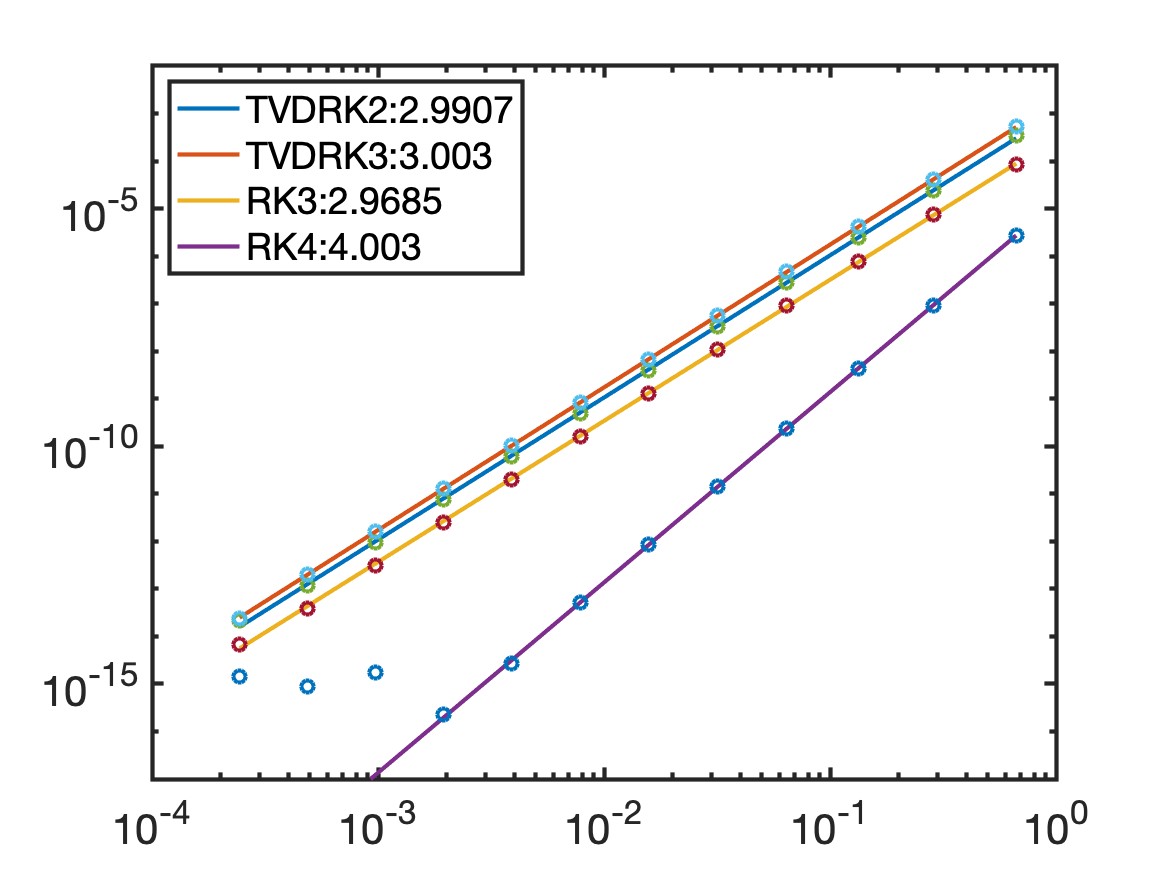}
{(b)}\includegraphics[trim=0 0 10 0, clip, width=0.45\textwidth]{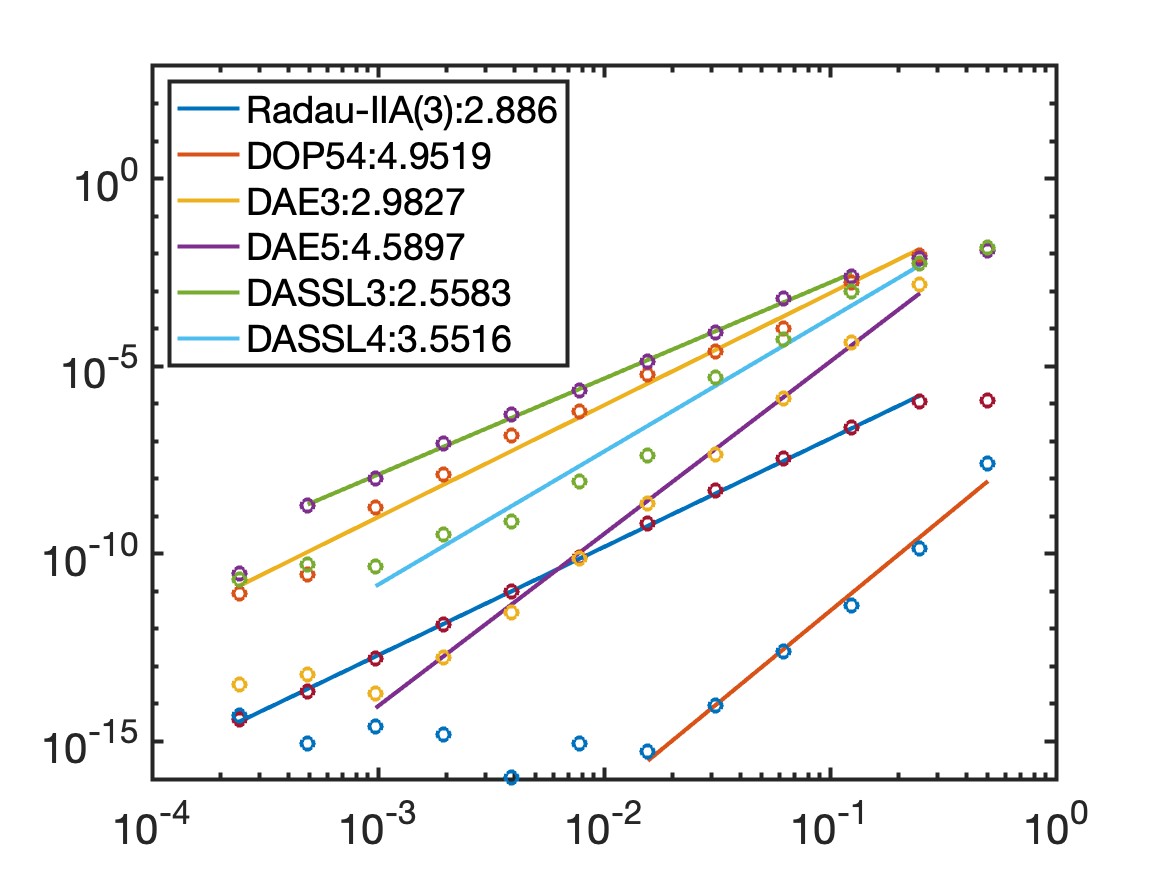}
\caption{(Section \ref{SubSec:ExConvergence}) (a) $E_{\text{norm}}$ errors in the solutions obtained by some explicit RK methods including TVDRK2, TVDRK3, RK3 and RK4. They demonstrate third or fourth-order convergence. {(b) $E_{\text{norm}}$ errors in the solutions obtained by some popular methods including Radau-IIA(3), DOP54 and DAE methods. They demonstrate third or fifth-order convergence.}
}
\label{Fig:ConvergenceNormError}
\end{figure}

To check the accuracy of the numerical solutions, we consider the following two error measures,
$$
E_2 = \| \p^N-\p_{\mbox{exact}} \|_2 \, \mbox{ and } \, E_{\mbox{norm}} = \left| \|\p^N\|-1 \right| 
$$
where $\mathbf{p}^N$ represents the solution at the final time $T=2$. The first error quantifies the $L_2$ discrepancy between the numerical solution and the exact solution of the ODE in $\mathbb{R}^3$. In cases where a convergent numerical scheme yields a solution not confined to $\mathbb{S}^2$ but rather situated in $\mathbb{R}^3$ generally, i.e., $\| \mathbf{p}^N \|$ may not maintain unit length for any finite $h$ value, this definition of error might not be the most crucial since solutions violating such a constraint may already lack physical interpretation across numerous applications. In light of this, we additionally quantify the error in the norm of the solution at the final time, which we term as $E_{\text{norm}}$.

The rate of convergence corresponding to these two errors for all the mentioned numerical integrators is presented in Table \ref{Table:Convergence}. An intriguing observation emerges from this analysis: the $E_{\text{norm}}$ error in the solution derived from the TVDRK2 scheme exhibits third-order convergence. Despite the fact that these standard numerical integrators do not explicitly account for the constraint, TVDRK2 appears to maintain the norm of the numerical solution better. Across most numerical schemes, the convergence in $E_2$ error aligns with the anticipated behavior. However, there is an exception with the PTVDRK3' scheme, where the observed convergence is one order lower than expected. Instead of the expected third-order convergence, we observe a second-order convergence. This example indicates that perturbations introduced by multiple projections can potentially introduce significant fluctuations in the overall accuracy of the numerical scheme. We refer to Appendix B for proof that this internal projection scheme PTVDRK3' is of second order.

Figures \ref{Fig:Convergence123}-\ref{Fig:ConvergenceNormError} illustrate the computed errors in the solutions obtained from all mentioned numerical integrators. In Figure \ref{Fig:Convergence123}(a), we have compiled the numerical integrators displaying first- or second-order accuracy. This group encompasses TVDRK2, PTVDRK2, PTVDRK2', PTVDRK3', SFE, and our proposed STVDRK2. The uppermost curve corresponds to SFE, which is an extension of the standard forward-Euler method tailored for spherical data. Consequently, observing first-order convergence in its error behavior is unsurprising. Five second-order TVDRK integrators are depicted in this figure. The standard TVDRK2 serves as the reference, with two PTVDRK2 variants stemming from it. As anticipated, these three methods exhibit consistent second-order convergence. Worth noting is that the PTVDRK3' algorithm is also second-order accurate. The lower curve corresponds to the least-squares fitting line for errors arising from our proposed STVDRK2 method. This method requires no velocity extension from the sphere or any intermediate projection steps, resulting in a solution of greater accuracy than other second-order methods. In Figure \ref{Fig:Convergence123}(b), the least-squares fitting lines for errors within three TVDRK-based integrators and Radau-IIA are displayed. With fewer excessive projection steps during intermediate stages, the PTVDRK3 method reinstates its third-order convergence property. Upon comparison of these three methods, it becomes evident that our proposed STVDRK3 method yields the most accurate solution. Figure \ref{Fig:Convergence123}(c-d) shows the numerical errors in the solution obtained from STVDRK4, SSSPRK(5,4), and SSSPRK(10,4) when the convex combination is performed using progressive SLERP and Fr\'echet mean, respectively. It can be observed that the numerical solutions do not exhibit the expected accuracy. Specifically, the accuracy decreases to third-order only when the convex combination is computed using progressive SLERP. Furthermore, when the Fr\'echet mean is implemented, the accuracy drops even further to second-order. {Figure \ref{Fig:Convergence123}(e-f) shows the error in the solutions obtained by some commonly used approaches with our proposed STVDRK3 as a reference. We consider the third-order Radau-IIA scheme, the DOP54 scheme, the third- and fifth-order DAE solvers, and the third- and fourth-order DASSL integrators. Although both the Radau-IIA(3) and the third-order DAE3 solvers provide third-order accurate solutions, the errors from these methods are more significant than those obtained by our proposed approach. While DAE5 offers roughly fifth-order convergence solutions, its error is smaller than that of STVDRK3 only when the time step is smaller than $O(10^{-2})$. DOP54 is a high-order numerical method that provides the most accurate solutions in this example.}

Figure \ref{Fig:ConvergenceNormError}{(a)} illustrates the convergence of the solution's norm using standard RK methods in the embedded space $\mathbb{R}^3$. Since no projection step is involved, these numerical methods are not sphere-invariant, implying that the normal in the numerical solution will generally not be unity. As mentioned earlier, TVDRK3 and RK3 exhibit third-order convergence in the norm of the solution, while the norm error in RK4 converges at fourth-order. The surprising result is that TVDRK2 also demonstrates a third-order convergence in the solution's norm. This result matches with the proof given in Appendix A. {Figure \ref{Fig:ConvergenceNormError}(b) shows the corresponding errors from some commonly used numerical integrators. We see that Radau-IIA(3), DOP54, DAE3, DAE5, DASSL3 and DASSL4 do not give exact constrained solutions. We observe that solutions from Radau-IIA(3) and DOP54 converge to the sphere with orders of 3 and 5, respectively. It is reasonable for solutions from these methods to deviate from the unit sphere since there is no mechanism to ensure these solutions satisfy the constraint. The geometric integrators DAE, however, do not strictly enforce the constraint but allow solution trajectories to leave the surface. This is because the constraint is a part of the DAE system. These implicit schemes involve solving a system of nonlinear equations using an iterative method, which accepts a numerical tolerance in the convergence criteria. We observe only a third- and fifth-order convergence in DAE3 and DAE5, respectively. The corresponding convergence in DASSL3 is between two and three, and that in DASSL4 is between three and four.}

%%%%%%%%%%

\begin{figure}[!htb]
\centering
(a-i)\includegraphics[trim=20 0 70 0, clip, width=0.4\textwidth]{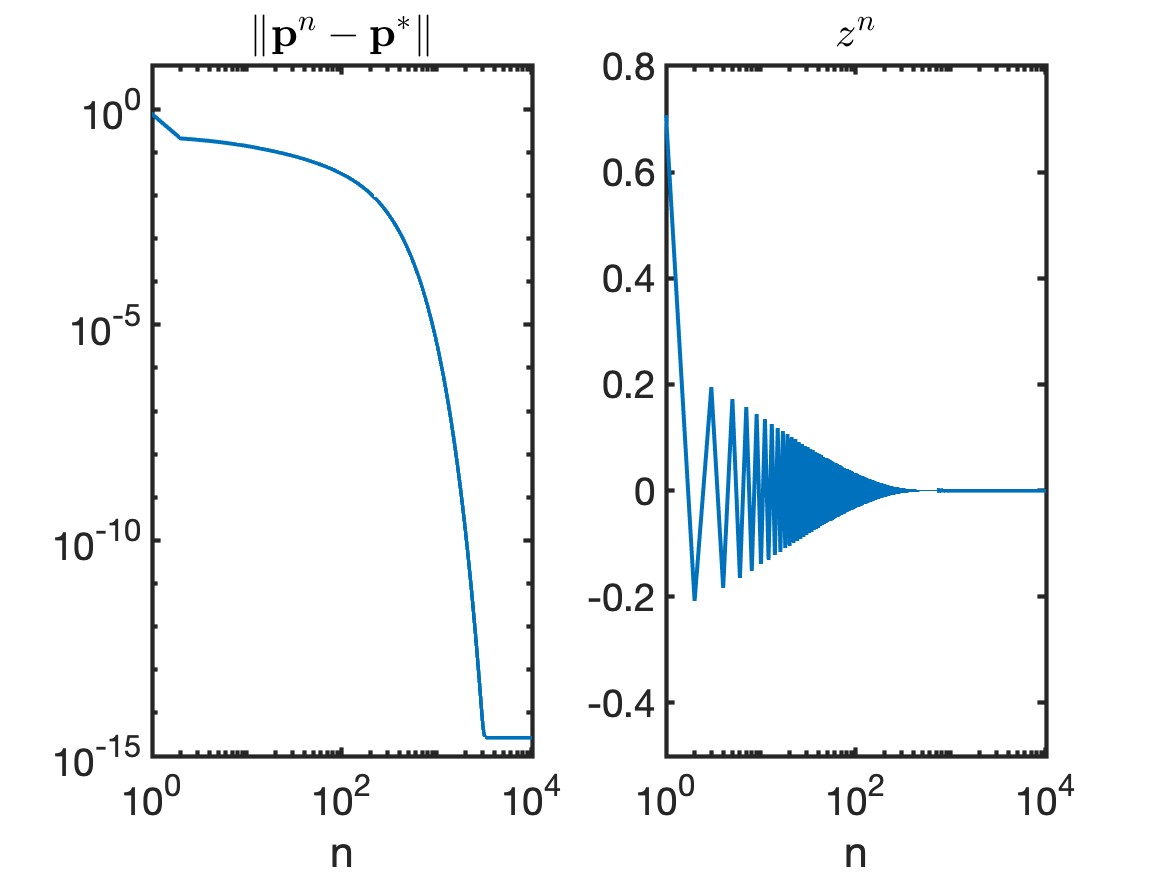}
(a-ii)\includegraphics[trim=20 0 70 0, clip, width=0.4\textwidth]{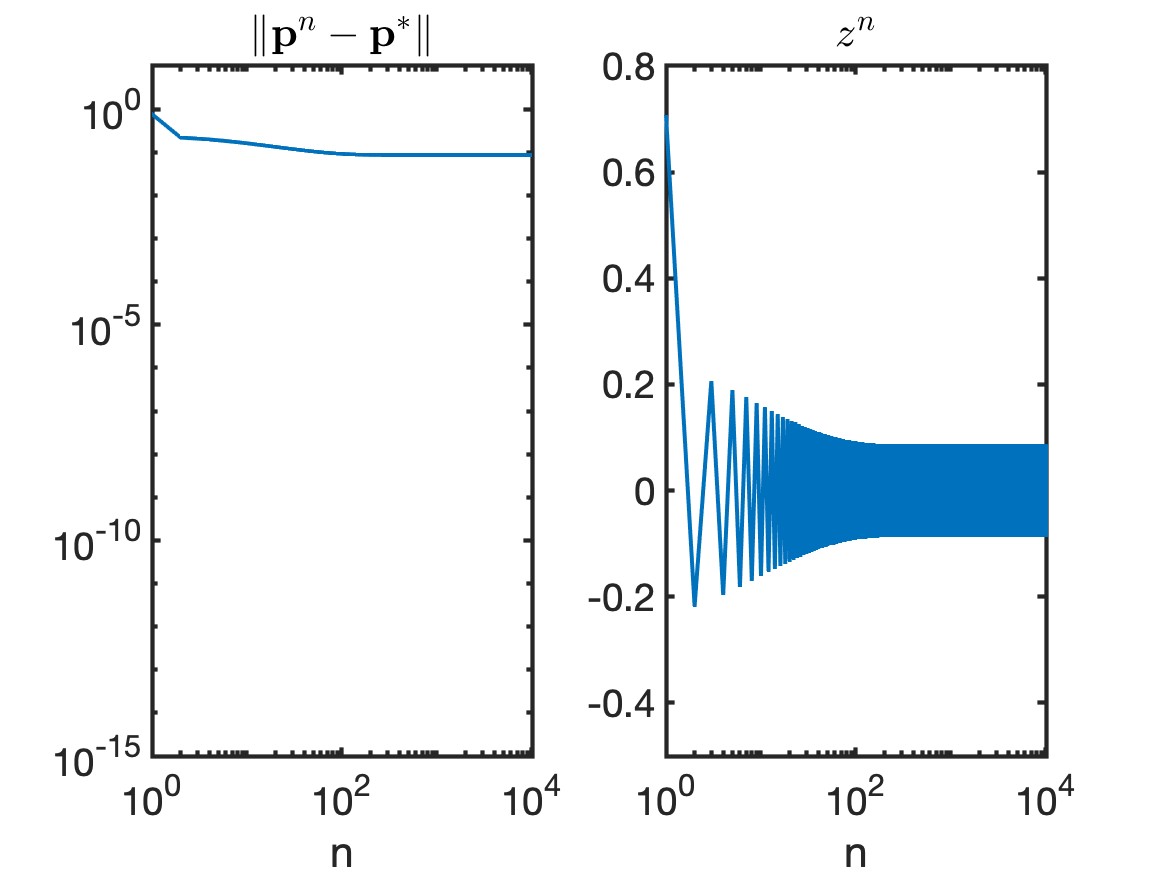} \\
(b-i)\includegraphics[trim=20 0 70 0, clip, width=0.4\textwidth]{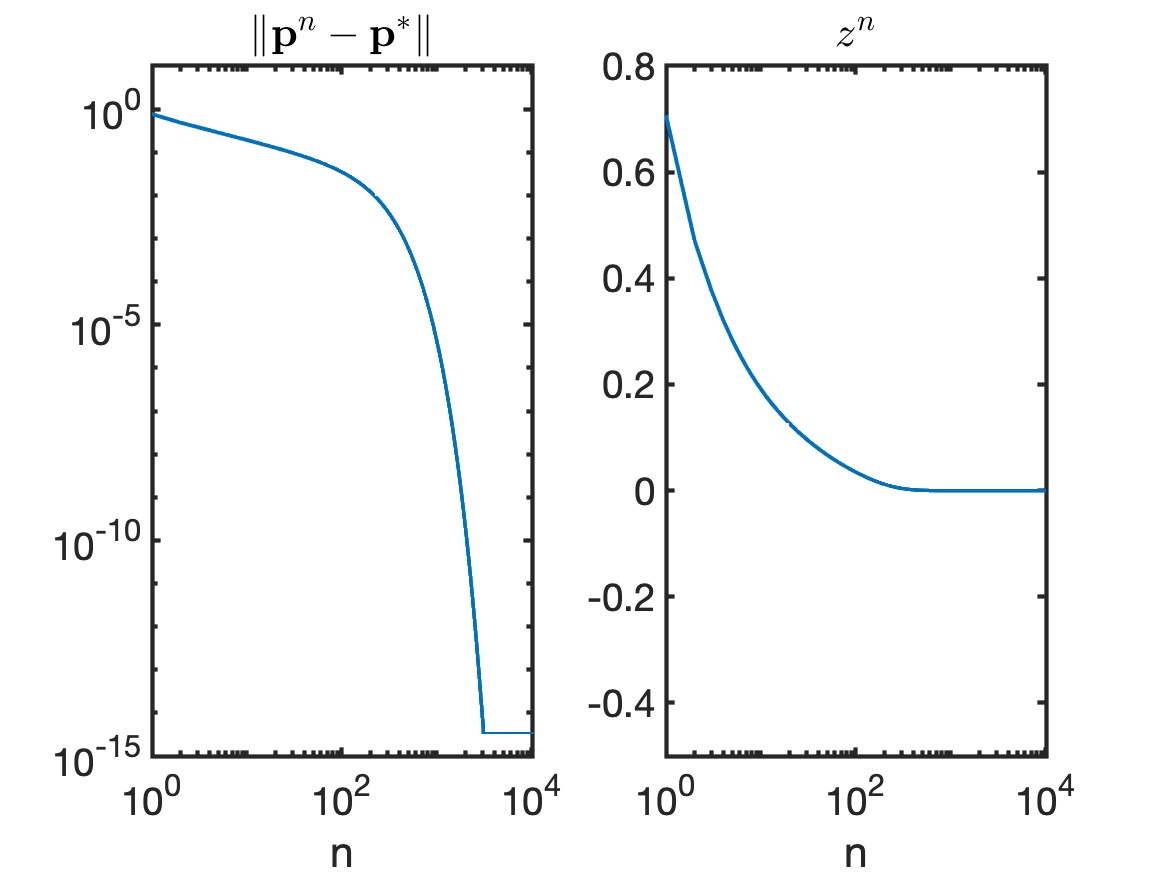}
(b-ii)\includegraphics[trim=20 0 70 0, clip, width=0.4\textwidth]{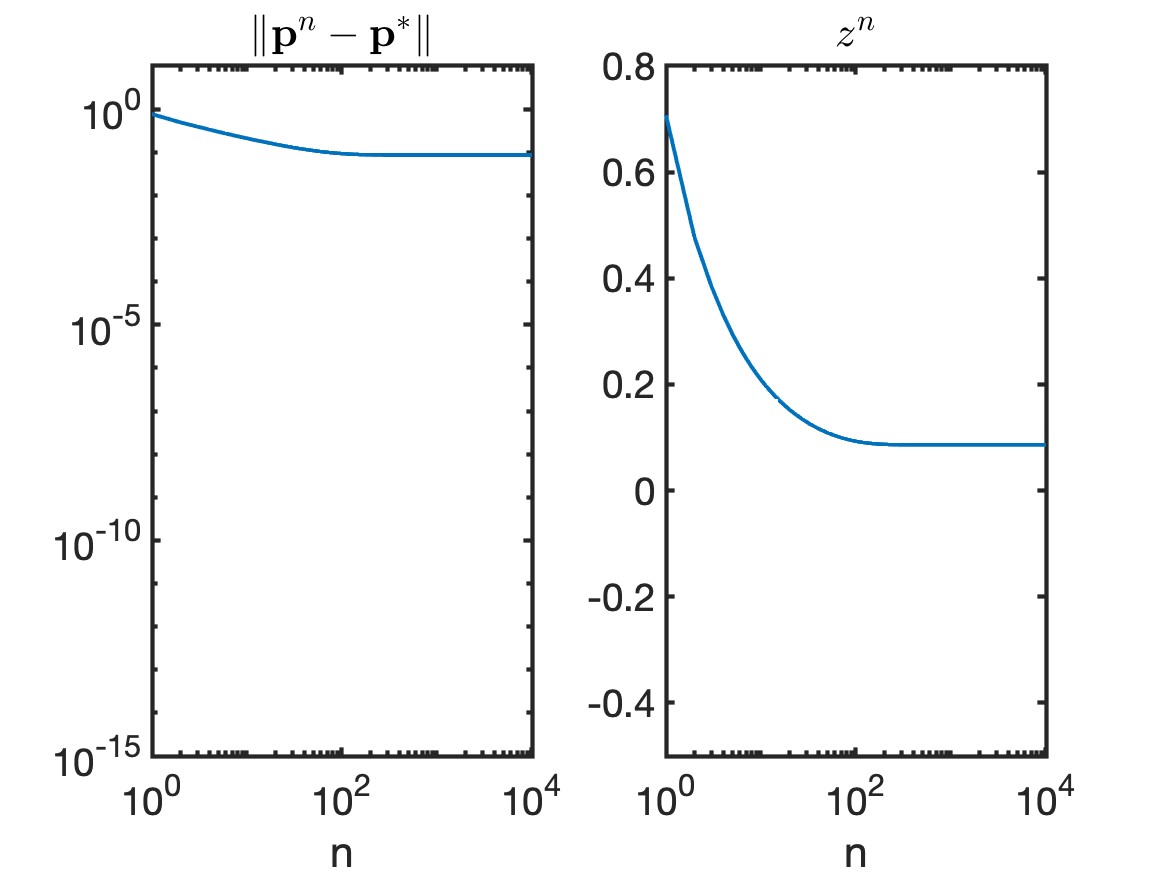} \\
(c-i)\includegraphics[trim=20 0 70 0, clip, width=0.4\textwidth]{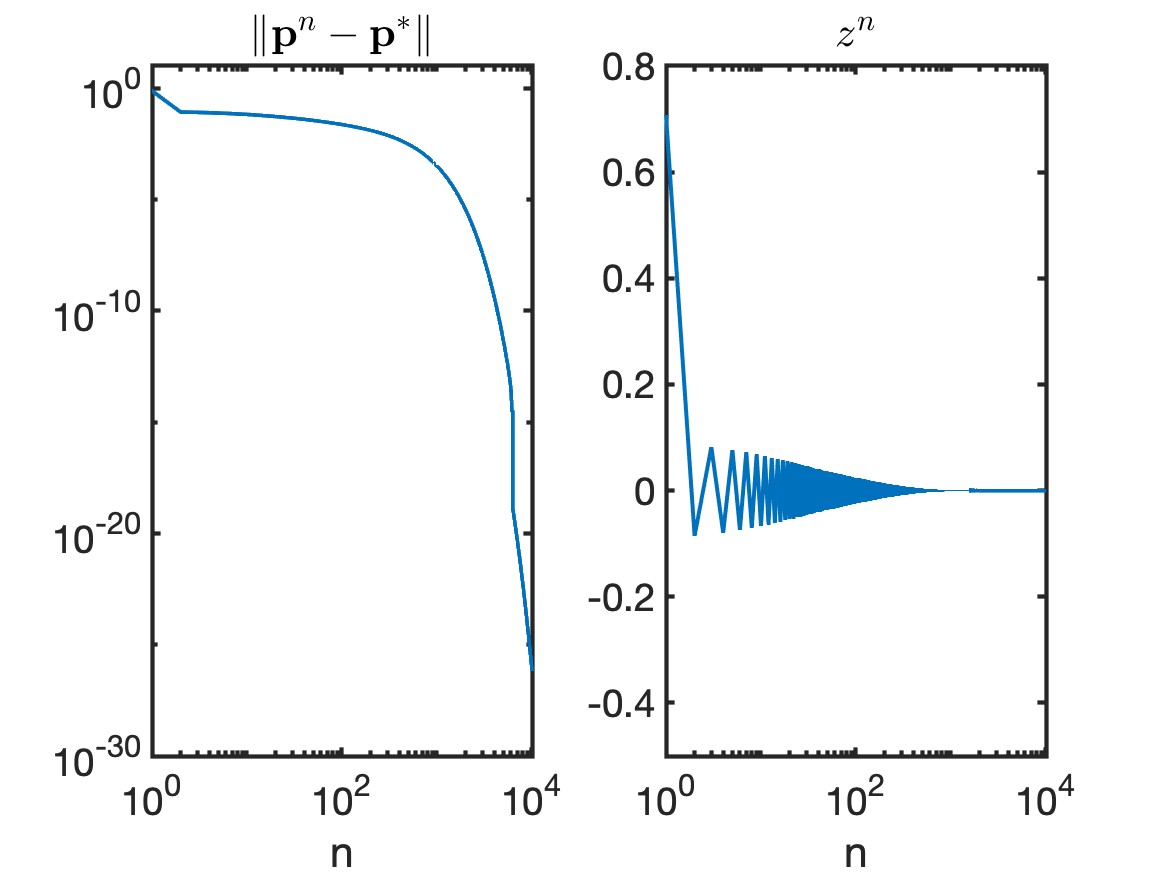}
(c-ii)\includegraphics[trim=20 0 70 0, clip, width=0.4\textwidth]{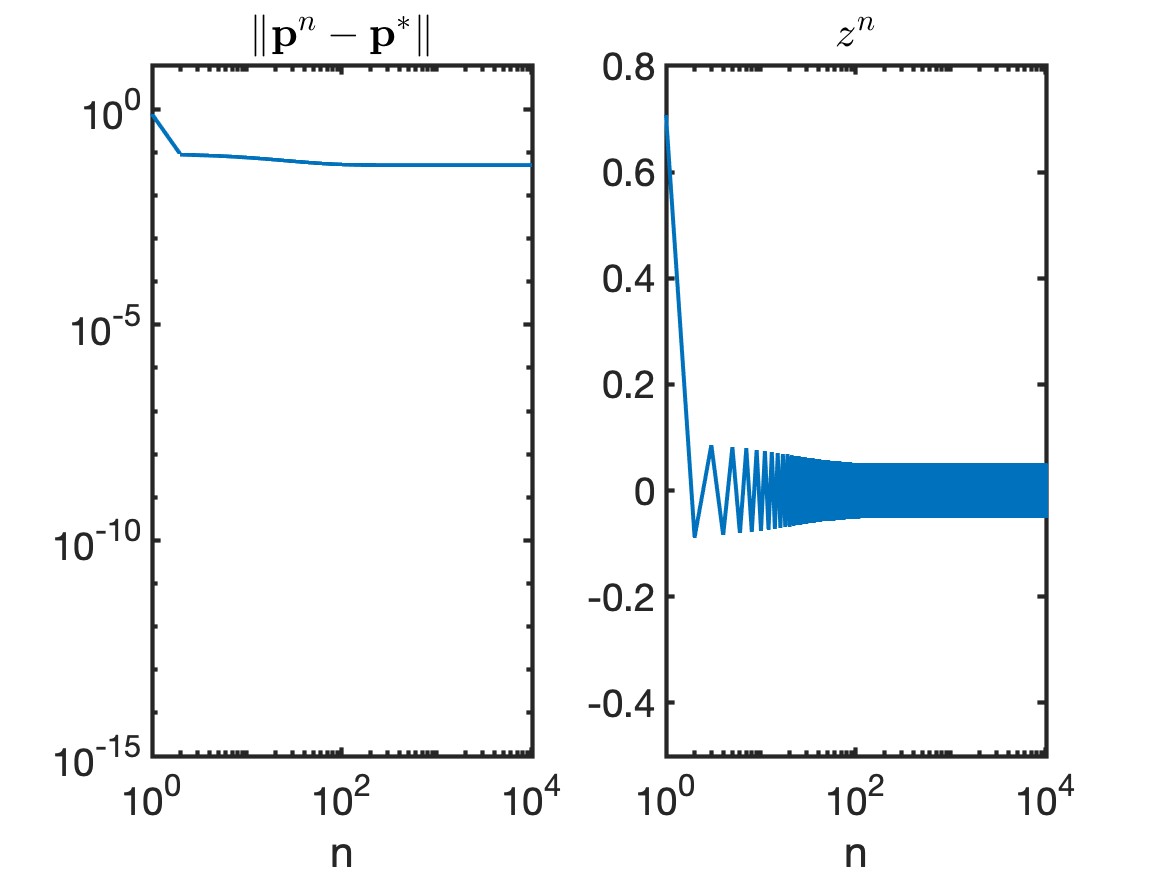} 
\caption{(Section \ref{Ex:Stability}) The numerical solutions obtained using (a) SFE, (b) STVDRK2, and (c) STVDRK3 are computed with time steps that both (i) satisfy and (ii) violate the A-stability condition. For the SFE and STVDRK2 methods, we adopt step sizes of 1.99 and 2.01. In the case of STVDRK3, we choose step sizes of 2.51 and 2.52.} 
\label{Fig:Stability}
\end{figure}

\subsection{Stability}
\label{Ex:Stability}

We demonstrate the significance of satisfying the A-stability-like condition in determining an appropriate time step size for a numerical method. As elaborated upon in Section \ref{SubSec:Stability}, we consider
$$
M=\left(
\begin{array}{ccc}
\frac{1}{2} & 0 & 0 \\
0 & -\frac{1}{2} &0 \\
0 & 0 & -\frac{1}{2}
\end{array}
\right) \, ,
$$
in the context of the linear ODE model problem (\ref{Eqn:LinearModelProblem}). This model problem possesses a sole equilibrium point at the origin, and the matrix $M$ is characterized by three eigenvalues. Among these, 1/2 and -1/2 each have a multiplicity of 2. The eigenvector associated with the eigenvalue 1/2 aligns with the $x$-direction (i.e., $\mathbf{e}_1$), contributing to the diverging component of the solution for the linear homogeneous problem. For the corresponding nonlinear ODE model problem (\ref{Eqn:NonlinearModelProblem}), the vectors $\mathbf{p} = \mathbf{e}_1$ and $-\mathbf{e}_1$ give rise to two equilibrium points on the unit sphere. The Jacobian matrix associated with these points on the tangent plane yields two eigenvalues of -1, indicating both equilibria are stable attractors. This highlights the importance of understanding the stability characteristics of numerical methods, especially when applied to problems with nonlinear dynamics.

Figure \ref{Fig:Stability} exhibits numerical solutions acquired through our proposed SFE, STVDRK2, and STVDRK3 methods employing two different step sizes. One step size adheres to the A-stability condition, while the other slightly exceeds the threshold. For instance, considering the SFE and STVDRK2 methods, we establish a time step size $h$ of less than 2. We perform simulations using $h=1.99$ and $2.01$, presenting the corresponding solutions in Figure \ref{Fig:Stability}(a1) and Figure \ref{Fig:Stability}(a2), respectively. In each figure, the left subplot illustrates the distance between intermediate solutions and the attractor $\mathbf{e}_1$ against the iteration number, while the right subplot depicts the third component of the solution ($z_3$). The solution with $h=1.99$ demonstrates favorable convergence towards the point $\mathbf{e}_1$, whereas the solution with $h=2.01$ diverges. For STVDRK3, we scrutinize solutions with $h=2.51$ and $2.52$, as illustrated in Figure \ref{Fig:Stability}(c). This analysis helps illustrate the effect of adhering to or deviating from the A-stability condition on the stability and convergence behavior of the numerical methods.

\subsection{Ray Tracing Solutions of the Surface Eikonal Equation}
\label{SubSec:SurfaceEikonal}

\begin{figure}[!htb]
\centering
\includegraphics[trim=40 0 80 0, clip, width=0.45\textwidth]{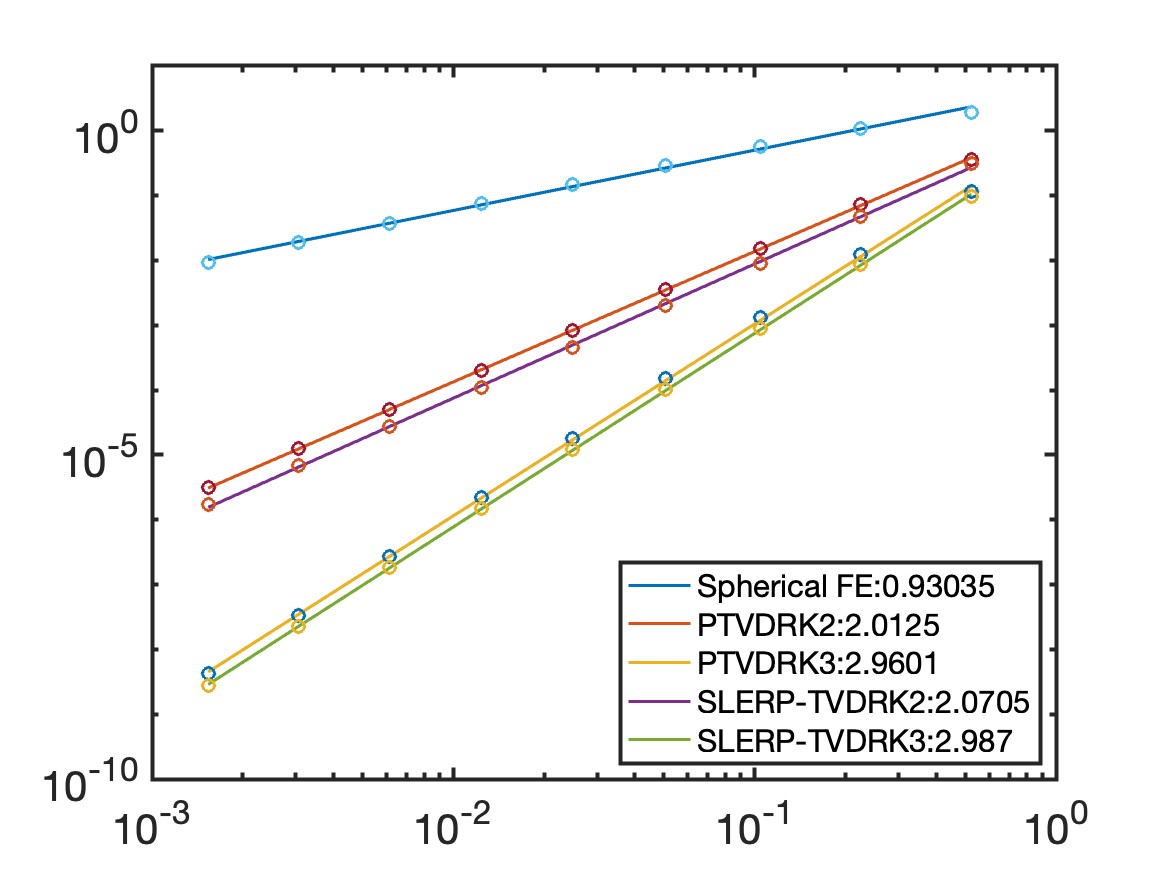}
\caption{(Section \ref{SubSec:SurfaceEikonal} with $v(\bx)=1$) $L_2$-errors in the solutions at $t=\pi/2$ obtained by various numerical integrators.}
\label{eik:converg}
\end{figure}

\begin{figure}[!htb]
\centering
\includegraphics[width=0.95\textwidth]{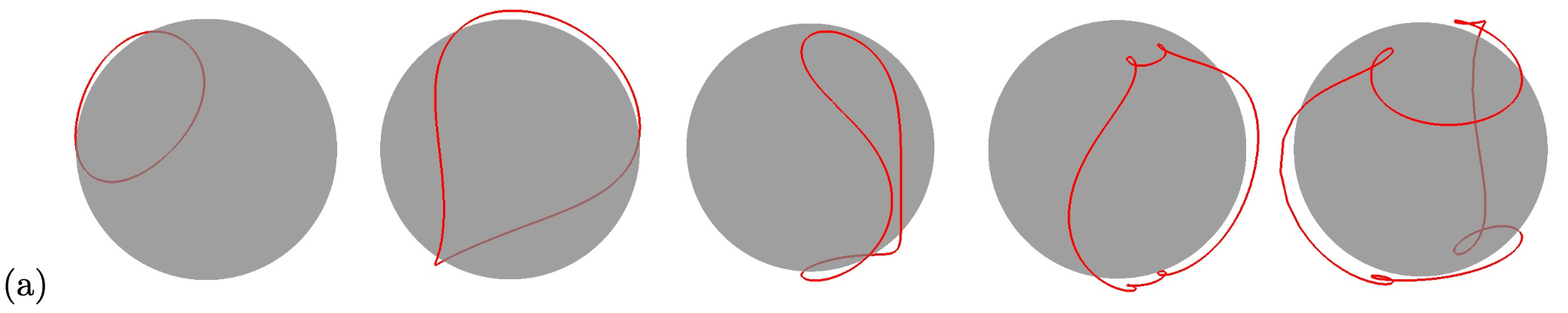} \\
\includegraphics[width=0.95\textwidth]{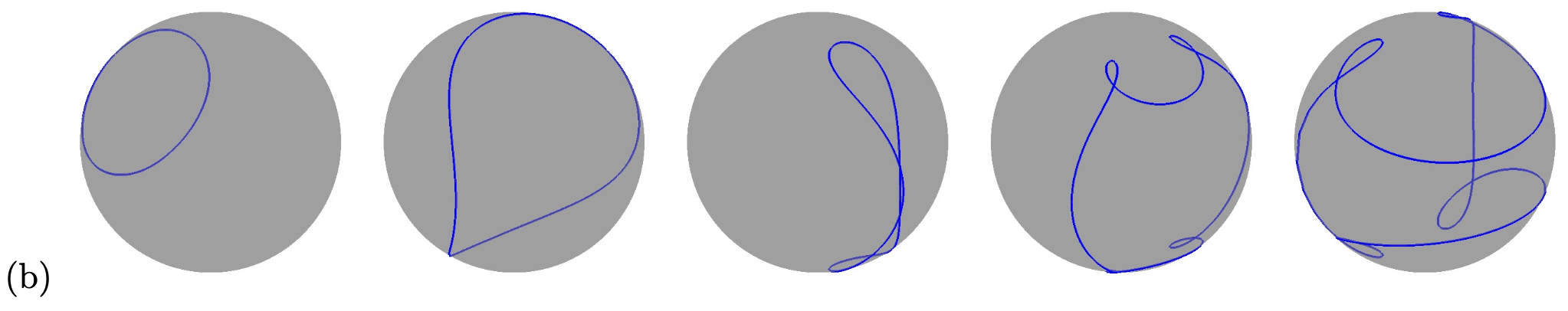} \\
{(c)} 
\includegraphics[width=0.54\textwidth]{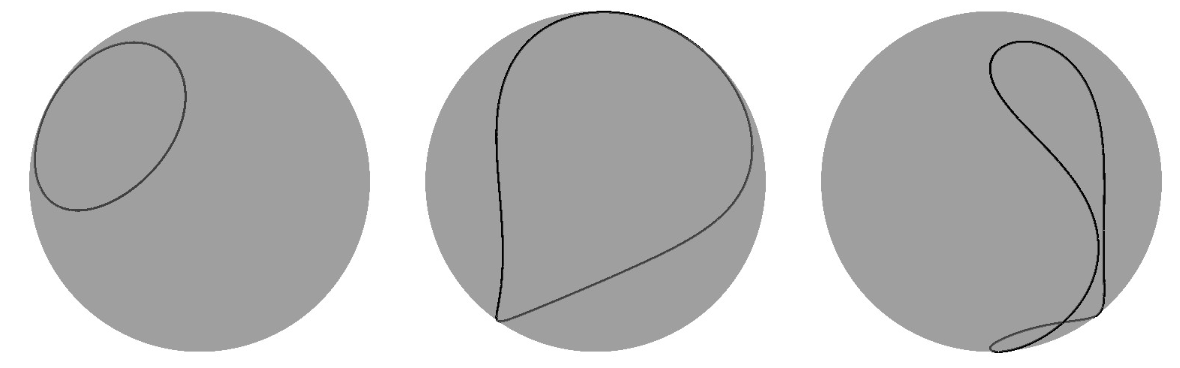} 
{(d)} 
\includegraphics[width=0.36\textwidth]{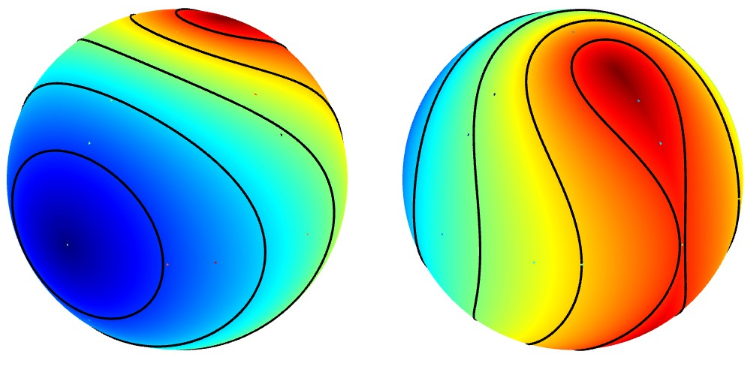} 
\caption{(Section \ref{SubSec:SurfaceEikonal} with $v(\bx)=\exp\left(-z^2\right)$) Wavefront at {$\Delta t$, $3\Delta t$, $5\Delta t$, $7\Delta t$ and $9\Delta t$ }with $\Delta t=\pi/5$ computed using (a) TVDRK3 for both $\bx$ and $\bp$, and (b) STVDRK3 for $\bx$ and TVDRK3 for $\bp$. {(c) Contour plots of the first-order accurate viscosity solution computed by \cite{wonleu16} with 641 grid points in each physical direction at  $\Delta t$, $3\Delta t$ and $5\Delta t$. (d) Contour plots of the first-order accurate viscosity solution computed by \cite{wonleu16} from two different angles.}
}
\label{eik:exponential}
\end{figure}

\begin{figure}[!htb]
\centering
%(a)
%\includegraphics[trim=350 200 310 160, clip, width=0.18\textwidth]{figures/Ex31_stvdrk3_75_1.jpg} 
%\includegraphics[trim=350 200 310 160, clip, width=0.18\textwidth]{figures/Ex31_stvdrk3_100_1.jpg} 
%\includegraphics[trim=350 200 310 160, clip, width=0.18\textwidth]{figures/Ex31_stvdrk3_125_1.jpg} 
%\includegraphics[trim=350 200 310 160, clip, width=0.18\textwidth]{figures/Ex31_stvdrk3_150_1.jpg} 
%\includegraphics[trim=350 200 310 160, clip, width=0.18\textwidth]{figures/Ex31_stvdrk3_175_1.jpg}  \\
\includegraphics[width=0.95\textwidth]{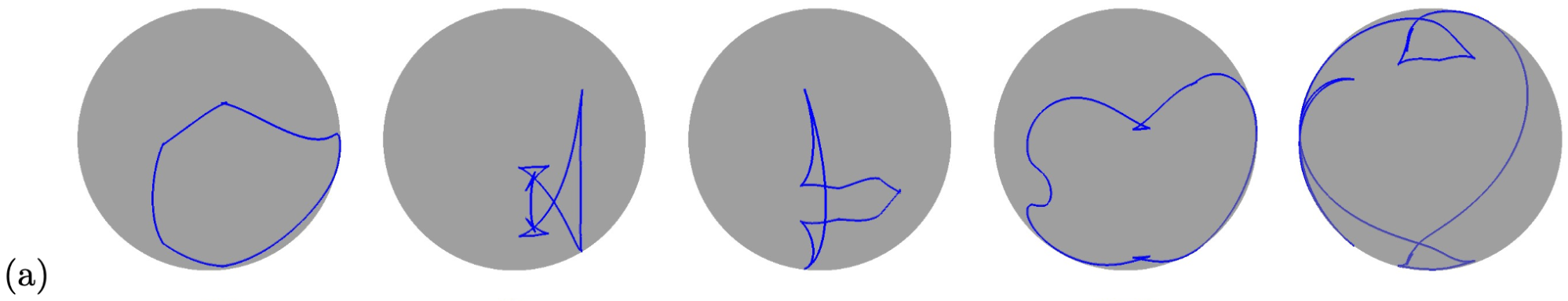} \\
%(b)
%\includegraphics[trim=330 200 280 160, clip, width=0.18\textwidth]{figures/Ex31_stvdrk3_75_2.jpg} 
%\includegraphics[trim=330 200 280 160, clip, width=0.18\textwidth]{figures/Ex31_stvdrk3_100_2.jpg} 
%\includegraphics[trim=330 200 280 160, clip, width=0.18\textwidth]{figures/Ex31_stvdrk3_125_2.jpg} 
%\includegraphics[trim=330 200 280 160, clip, width=0.18\textwidth]{figures/Ex31_stvdrk3_150_2.jpg} 
%\includegraphics[trim=330 200 280 160, clip, width=0.18\textwidth]{figures/Ex31_stvdrk3_175_2.jpg} \\
\includegraphics[width=0.95\textwidth]{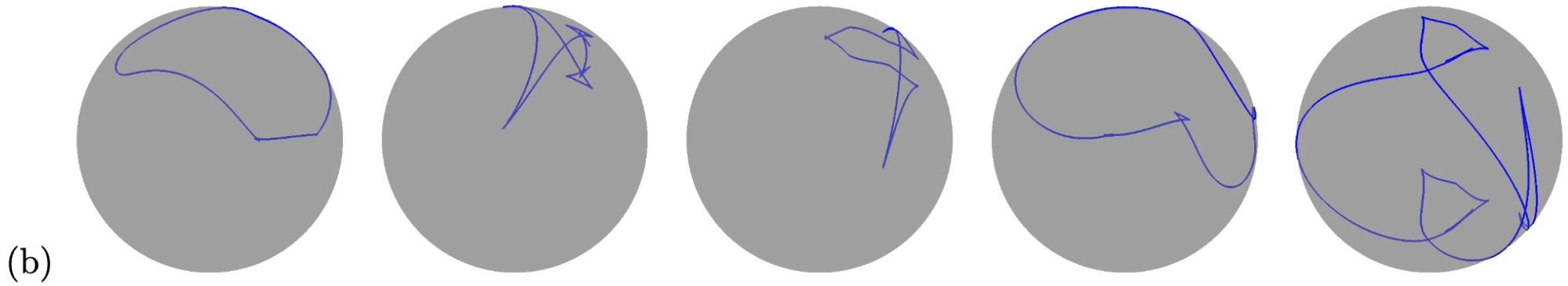} \\
{(c)} 
\includegraphics[width=0.54\textwidth]{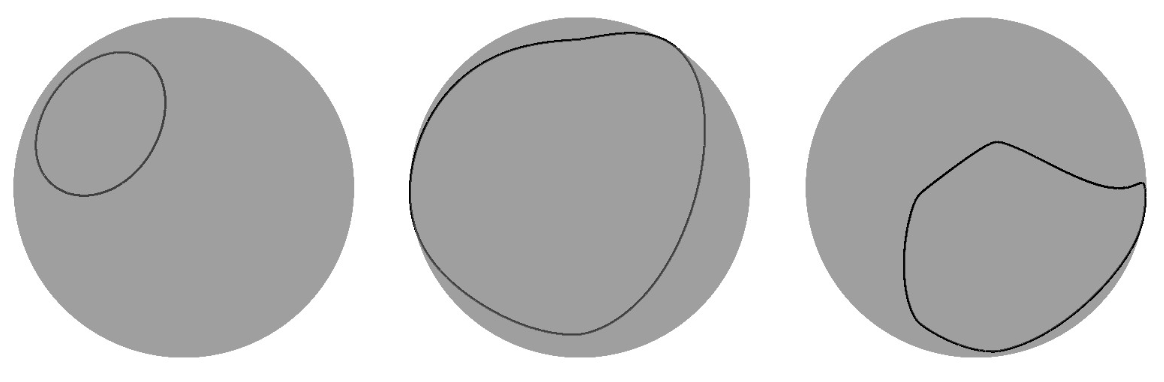}
{(d)} 
\includegraphics[width=0.36\textwidth]{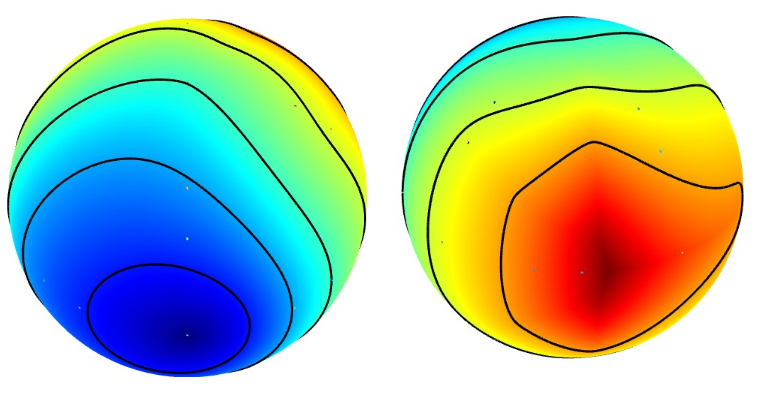} 
\caption{(Section \ref{SubSec:SurfaceEikonal} with $v(\bx)=1+Y_3^1(\theta,\rho)$) Wavefront at $75\Delta t$, $100\Delta t$, $125\Delta t$, $150\Delta t$ and $175\Delta t$ with $\Delta t=\pi/100$ computed using STVDRK3 for $\bx$ and TVDRK3 $\bp$. (a) and (b) are the solutions from different angles. {(c) Contour plots of the first-order accurate viscosity solution computed by \cite{wonleu16} with 641 grid points in each physical direction at $15\Delta t$, $45\Delta t$, and $75\Delta t$. (d) Contour plots of the first-order accurate viscosity solution computed by \cite{wonleu16} from two different angles.}
}
\label{eik:Y31}
\end{figure}

%\begin{figure}[!htb]
%\includegraphics[trim=10 80 30 50, clip, width=0.49\textwidth]{figures/vez^t_ft_left.jpg}
%\includegraphics[trim=10 80 30 50, clip, width=0.49\textwidth]{figures/vez^t_ft_right.jpg}
%\caption{ The ray solutions computed at $t=3\pi/2$ for $\Delta t =0.5 $ with $v(\bx)=e^{-z^2}$.}
%\label{eik:left_right}
%\end{figure}
%		
%		
%
%\begin{figure}[!htb]
%\includegraphics[trim=10 80 30 50, clip, width=0.49\textwidth]{figures/eik_constant.jpg}
%\includegraphics[trim=10 80 30 50, clip, width=0.49\textwidth]{figures/eik_e.jpg}
%\caption{ The ray solutions computed at $t=\pi/2$ for $\Delta t =0.5 $(a) $v(\bx)=1$, (b) $v(\bx)=e^{-z^2}$}
%\label{eik:fig_converg_graph}
%\end{figure}

We consider the surface eikonal equation and develop a ray-tracing solver for high-frequency wave propagation. Our previous work, such as \cite{qialeu04,leuqiaosh0401,qialeu06,leuqiabur07}, has provided efficient numerical solvers for computing multivalued solutions to the eikonal equation. However, these numerical solutions are typically defined in Cartesian space. In this work, our main focus is on numerical solutions defined on a unit sphere. This is crucial for understanding wave motions on the Earth's surface, as it allows us to accurately model wave propagation phenomena in geophysical contexts. By considering the solutions on a sphere, we aim to capture the unique characteristics and behaviors of surface waves, providing valuable insights into their propagation patterns. The surface eikonal equation on a unit sphere \cite{gri68} is given by:
\begin{equation}
\left\| \nabla_{\mathbb{S}^2} u(\mathbf{x}) \right\| = \frac{1}{v(\mathbf{x})},
\label{Eqn:surf_eikonal}
\end{equation}
for $\mathbf{x} \in \mathbb{S}^2 \backslash P$, with the boundary condition $u(\mathbf{x}_s) = 0$ for $\mathbf{x}_s \in P$, where $P$ is the set of sources, and $v(\mathbf{x})$ is the wave velocity. Similar to the typical eikonal equation, the solution $u(\mathbf{x})$ denotes the travel time of the wave from the source $\mathbf{x}_s$ to a target point $\mathbf{x} \in \mathbb{S}^2$. The surface gradient $\nabla_{\mathbb{S}^2}$ is defined as the orthogonal projection of the usual gradient $\nabla$ onto the tangent plane of the sphere, i.e., $\nabla_{\mathbb{S}^2} u(\mathbf{x}) = \nabla u(\mathbf{x}) - \left[ \mathbf{n} \cdot \nabla u(\mathbf{x}) \right] \mathbf{n}$, where $\mathbf{n}$ is defined as the outward unit normal of the sphere given by $\mathbf{n} = \frac{\mathbf{x}}{\| \mathbf{x} \|}$. The corresponding Hamiltonian of the surface eikonal equation is given by $H(\mathbf{x}, \mathbf{k}, u) = \frac{1}{2} \left\{v(\mathbf{x})^2 \left[ \|\mathbf{k}\|^2 - (\mathbf{k} \cdot \mathbf{n})^2 \right] - 1\right\}$ which leads to the following ray tracing system:
\begin{align*}
\mathbf{x}'(t) &=f_1(\mathbf{x},\mathbf{k}) := \nabla_{\mathbf{k}} H = v(\mathbf{x})^2 \left[ \mathbf{k} - (\mathbf{x} \cdot \mathbf{k}) \frac{\mathbf{x}}{\| \mathbf{x} \|} \right]  \, , \\
\mathbf{k}'(t) &=f_2(\mathbf{x},\mathbf{k}) := -\nabla_{\mathbf{x}} H = \frac{v(\mathbf{x})^2 (\mathbf{x} \cdot \mathbf{k})}{\| \mathbf{x} \|} \left[ \mathbf{k} - \frac{(\mathbf{x} \cdot \mathbf{k})}{\| \mathbf{x} \|} \mathbf{x} \right] - \frac{1}{v(\mathbf{x})} \nabla v(\mathbf{x}) 
\, , \\
u'(t) &= \mathbf{k} \cdot \nabla_{\mathbf{k}} H = 1 \, .
\end{align*}
The third equation guarantees that the solution $u$ can be used as the parameterization along the ray and has the physical meaning as the phase of the high-frequency wave. The first two equations represent the location of the ray on the unit sphere at a specific time, and the vector $\mathbf{k}$ denotes the ray direction $\nabla u$ along the ray trajectory.

Numerically, since the trajectory $\mathbf{x}(t) \in \mathbb{S}^2$, we apply the STVDRK method to integrate the ODE. However, the gradient is generally a vector in $\mathbb{R}^3$. Therefore, we will use the TVDRK method with the corresponding order to update the gradient $\mathbf{k}(t)$. Taking the second-order scheme as an example, we have the following combination of STVDRK2 and TVDRK2 for $\mathbf{x}$ and $\mathbf{k}$, respectively,
\begin{align*}
	\begin{cases}
		\bq_1=\text{exp}_{\bx^n}(hf_1(\bx^n,\bk^n)) \\
		\bs_1=\bk^n+hf_2(\bx^n,\bk^n) \\
		\bq_2=\text{exp}_{\bq_1}(hf_1(\bq_1,\bs_1)) \\
		\bs_2=\bs_1+hf_2(\bq_1,\bs1) \\
		\bx^{n+1}=\text{SLERP}(\bx^n,\bq_2,0.5)\\
		\bk^{n+1}=\frac{1}{2}(\bk^n + \bs_2) \, .
	\end{cases}
\end{align*}

We consider a point source $\bx_s$ at (1,0,0) in the following numerical examples. To test the convergence of various schemes, we first assume $v(\bx)=1$ and $t=\frac{\pi}{2}$, and define the error measure $E_2=\left[ \int_C \left(\frac{\pi}{2}-d(\bx,\bx_s) \right)^2 dS \right]^{1/2}$. Here, $C$ represents the wavefront where $u(\bx)=\pi/2$, and the function $d(\cdot,\cdot)$ denotes the geodesic distance between two points on $\mathbb{S}^2$. Figure \ref{eik:converg} illustrates the measured errors in the ray tracing solutions obtained by SFE, PTVDRK2, PRVDRK3, STVDRK2, and STVDRK3. As expected, the integrator SFE exhibits first-order convergence behavior in its error, indicated by the top line. PTVDRK2 and STVDRK2 demonstrate clear second-order convergence, with STVDRK2 producing a more accurate solution than PTVDRK2. Finally, it is unsurprising that both PTVDRK3 and STVDRK3 display third-order convergence. In general, when PTVDRK and STVDRK methods have the same order of convergence, STVDRK methods provide more accurate solutions.

Figure \ref{eik:exponential} illustrates an inhomogeneous case where the velocity on the sphere is defined as $v(\bx)=\exp\left(-z^2\right)$. We compute the solution using two different third-order methods up to $t=2\pi$ with a relatively large timestep of $\Delta t=\pi/5$. Figure \ref{eik:exponential}(a) displays the solutions computed by TVDRK3 without any projection. The wavefront departs from the sphere right from the beginning, even when the curve is relatively smooth. In Figure \ref{eik:exponential}(b), we apply the STVDRK3 method to the position. By default, the wavefront remains on the sphere at all times.
{Additionally, we have also compared the ray-tracing solution with viscosity solutions computed using the first-order fast-sweeping method developed in \cite{wonleu16}. In Figure \ref{eik:exponential}(c-d), we present contour plots of the complete viscosity solution with the contour levels $t=\Delta t$, $3\Delta t$, and $5\Delta t$ and a colored contour plot of the whole viscosity solution. At $t=3\Delta t$, the ray-tracing solution closely aligns with the viscosity solutions computed with a fine mesh. As time progresses, the viscosity solution from the fast-sweeping method fails to capture the multi-valued solution developed at around $t=5\Delta t$. In contrast, the ray-tracing method can effectively capture all multiple arrival wavefronts even after $t=7\Delta t$.
}

Next, we consider a more complex velocity model given by $v(\bx)=1+Y_3^1(\theta,\varphi)$, where the function
$
Y_3^1(\theta,\varphi) = \frac{-1}{8}\sqrt{\frac{21}{\pi}} \cos\varphi \sin \theta \left( 5 \cos^2\theta-1 \right)
$
represents a spherical harmonic. Here, $(\theta,\varphi)$ denotes the spherical and azimuthal angles in the spherical coordinate representation of the point $\bx$. We apply STVDRK3 with a relatively small timestep of $\Delta t=\pi/100$ and obtain the wavefront solution at the final time of $1.75\pi$. Figure \ref{eik:Y31} presents the solution at different times from two angles. The wavefront's self-intersection is accurately captured, and we observe the development of sharp structures within the wavefront. {Similarly, we show the viscosity solutions computed using the first-order fast-sweeping method \cite{wonleu16} in Figure \ref{eik:Y31}(c-d). The viscosity solutions appear smoother than those obtained through the ray-tracing method. In contrast, the fast-sweeping method cannot capture the wavefront after $t=100 \Delta t$ representing the later arrival information. 
}

%%%%%%%%%%%%%%%%%%%

\begin{figure}[!htb]
\centering
(a) \includegraphics[trim=100 20 80 20, clip, width=0.45\textwidth]{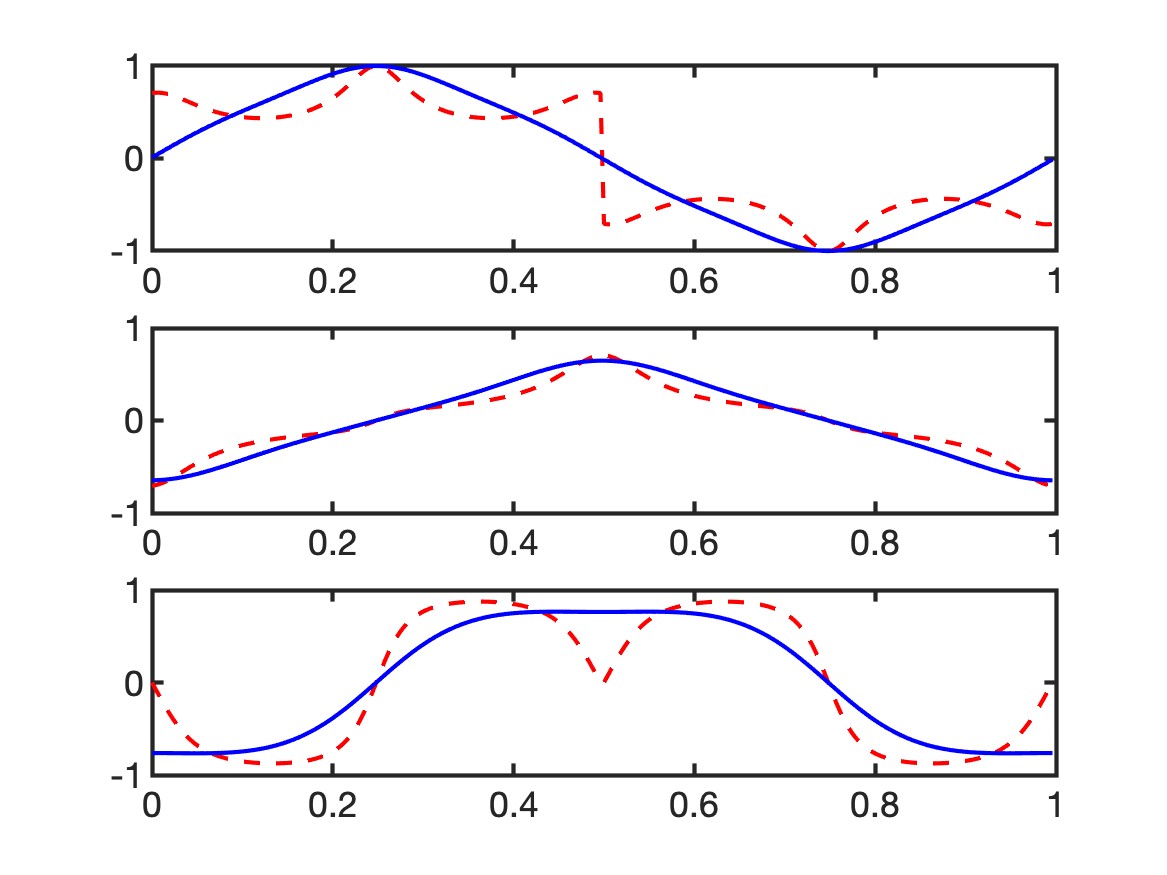}
\includegraphics[trim=100 20 80 20, clip, width=0.45\textwidth]{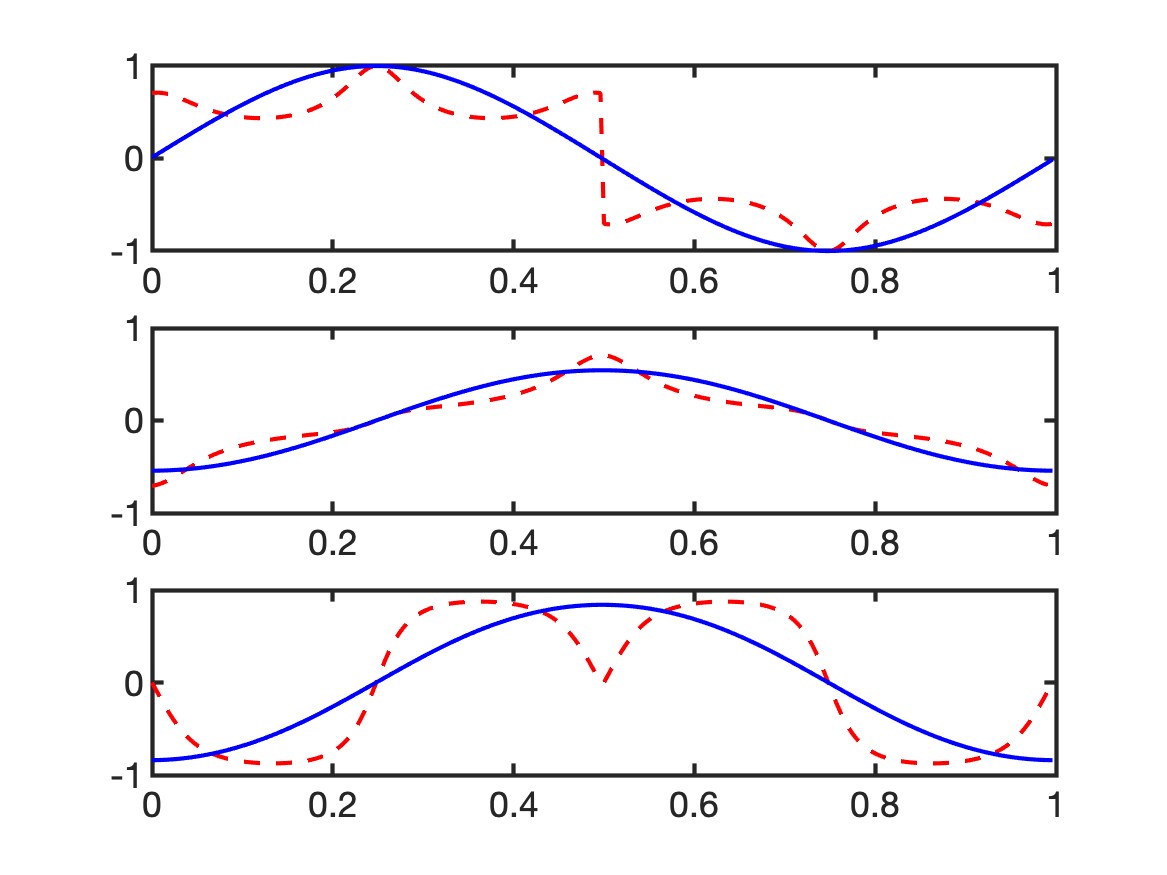} \\
(b) \includegraphics[trim=180 20 160 20, clip, width=0.42\textwidth]{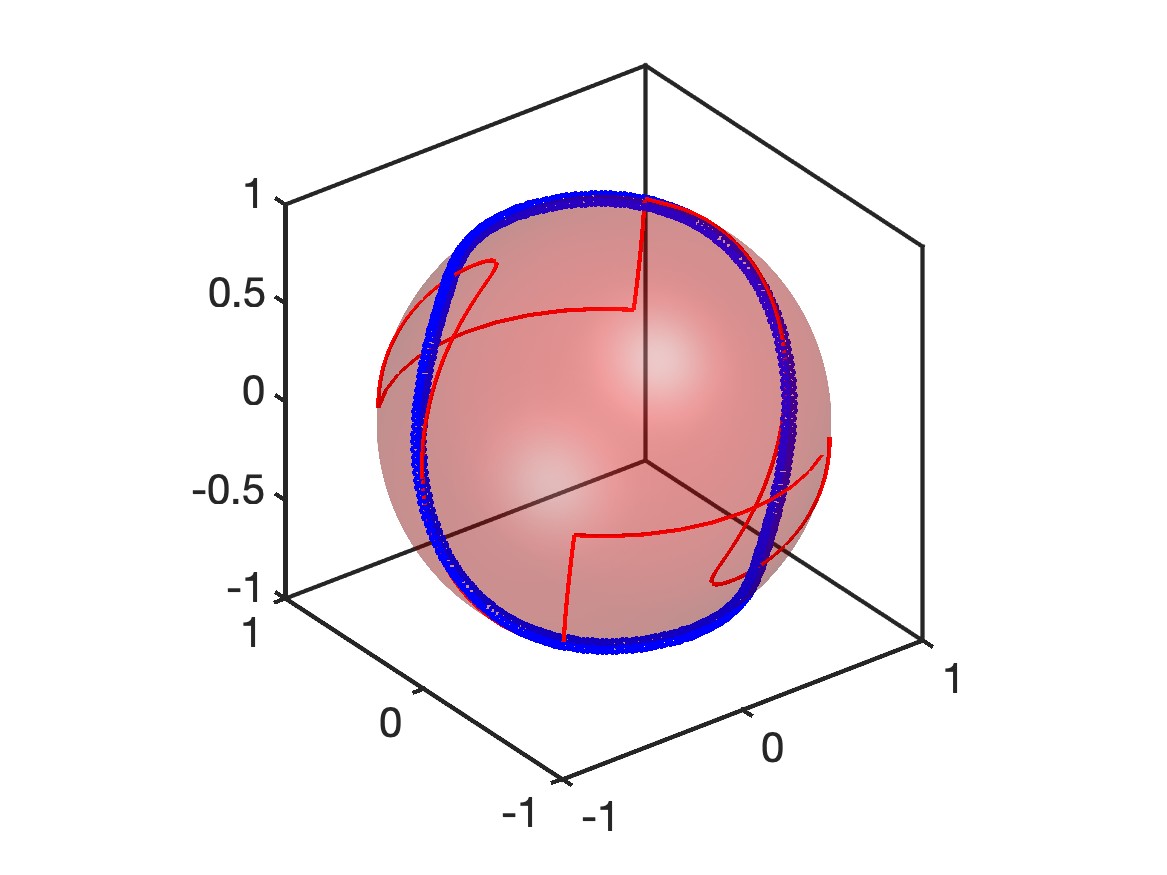}
\includegraphics[trim=180 20 160 20, clip, width=0.42\textwidth]{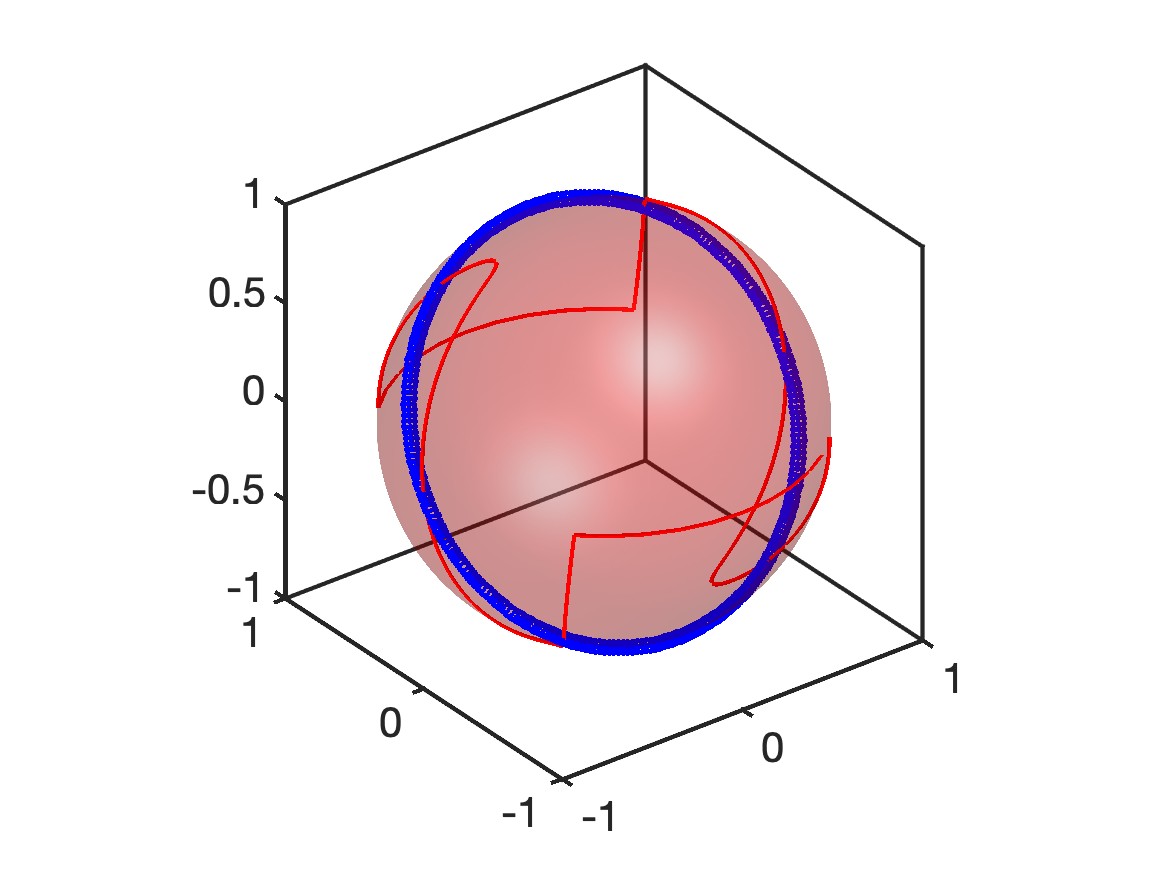}
\caption{(Section \ref{SubSec:ExpHarmonic}) STVDRK3 scheme for the $p$-harmonic flow is shown for $p=2$ at two different time levels. The initial condition is a discontinuous curve on the unit sphere, indicated by the red dashed lines. The solutions are plotted using blue circles connected by solid lines. (a) Individual component of $\m$. (b) Solution on the unit sphere.}
\label{Fig:DiscontinuousP2}
\end{figure}

\begin{figure}[!htb]
\centering
(a) \includegraphics[trim=100 20 80 20, clip, width=0.45\textwidth]{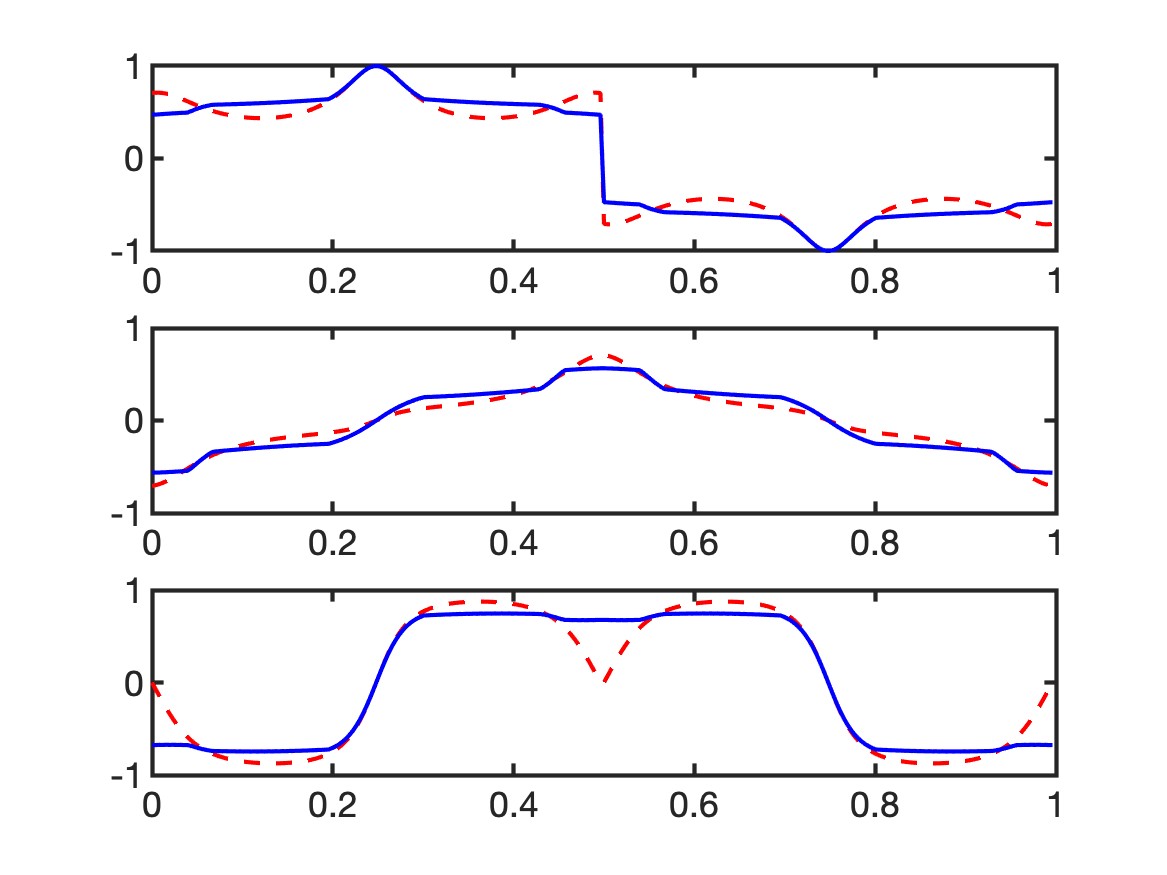}
\includegraphics[trim=100 20 80 20, clip, width=0.45\textwidth]{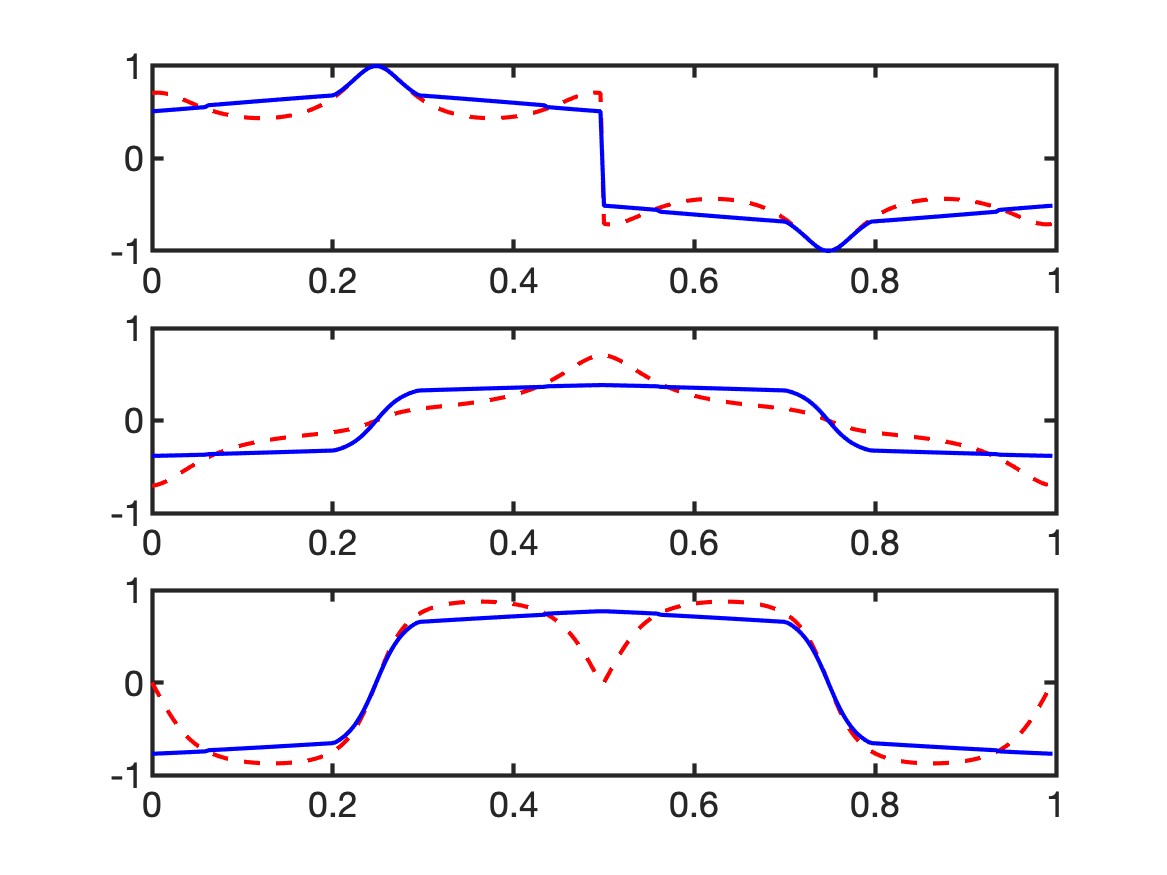} \\
(b) \includegraphics[trim=180 20 160 20, clip, width=0.42\textwidth]{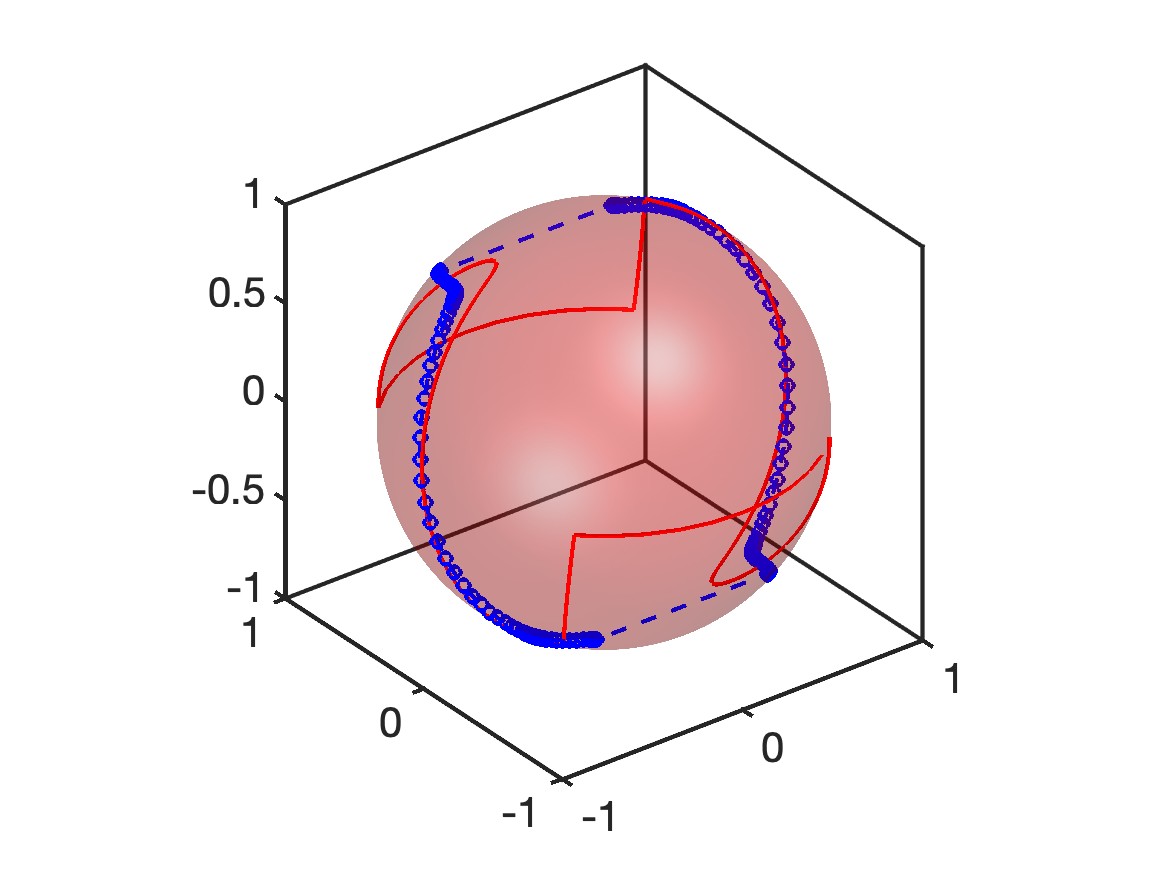}
\includegraphics[trim=180 20 160 20, clip, width=0.42\textwidth]{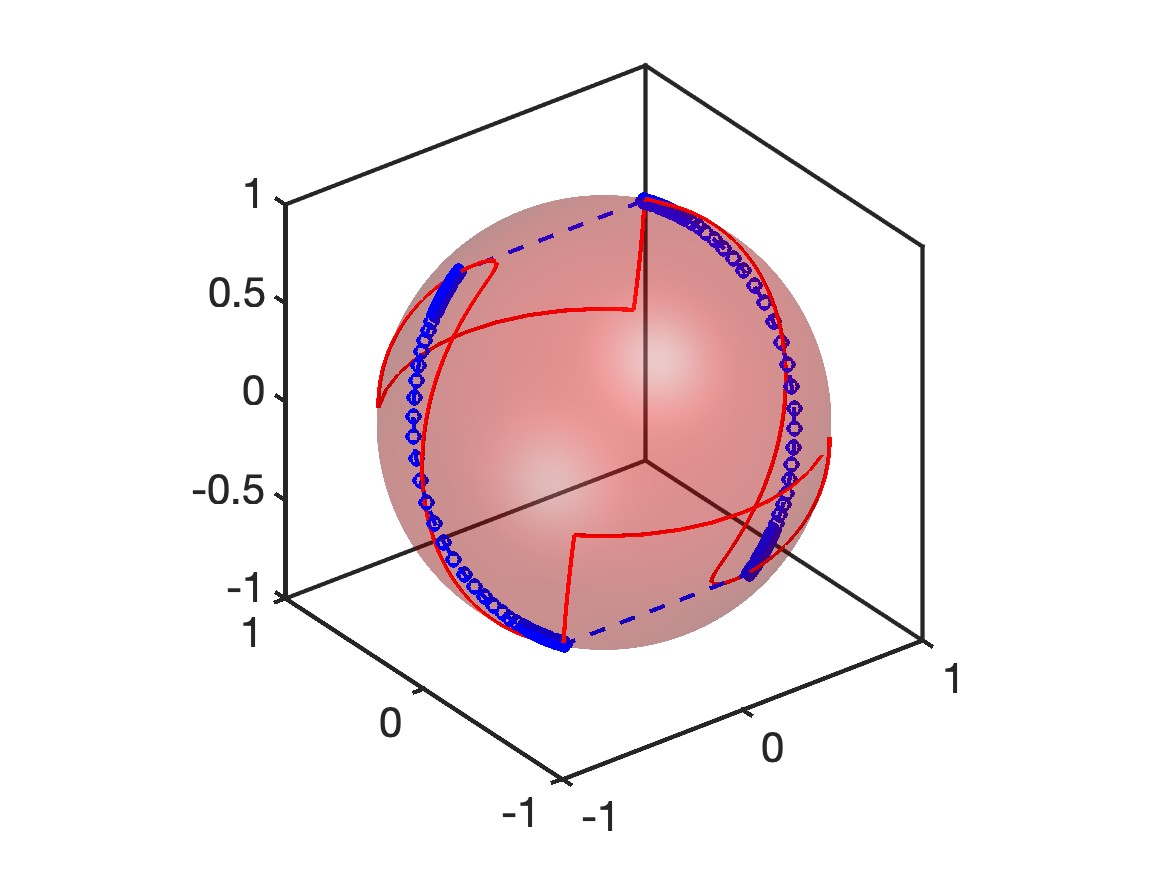}
\caption{(Section \ref{SubSec:ExpHarmonic}) STVDRK3 scheme for the $p$-harmonic flow is shown for $p=1$ at two different time levels. The initial condition is a discontinuous curve on the unit sphere, indicated by the red dashed lines. The solutions are plotted using blue circles connected by dashed lines. (a) Individual component of $\m$. (b) Solution on the unit sphere.}
\label{Fig:DiscontinuousP1}
\end{figure}

\subsection{$p$-Harmonic Flows}
\label{SubSec:ExpHarmonic}

The $p$-energy is given by $E_p(\mathbf{m}) = \frac{1}{p} \int_{\Omega} \| \nabla \mathbf{m} \|^p \, dV$ where $p \ge 1$. The vector $\mathbf{m}: \Omega \subset \mathbb{R}^m \rightarrow \mathbb{S}^{n-1}$ is a vector-valued function with $m \ge 1$ and $n \ge 2$ such that $\|\mathbf{m}\|=1$ almost everywhere $x \in \Omega$. The gradient of $\mathbf{m}$ is interpreted as
$$
\nabla \m = \left(
\begin{array}{c}
\nabla r_1 \\
\vdots \\
\nabla r_n 
\end{array}
\right)
\,
\mbox{ and the norm of $\nabla \m$ is given by }
\, 
\| \nabla \m \| = \sqrt{ \sum_{i=1}^n \sum_{j=1}^m \left( \frac{\partial r_i}{\partial x_j} \right)^2 } \, .
$$
The critical point of this $p$-energy is a map $\mathbf{m} \in C^1(\Omega, \mathbb{S}^{n-1})$ that solves the Euler-Lagrange equation $-\Delta_p \mathbf{m} = \|\nabla \mathbf{m}\|^p \mathbf{m}$, where $\Delta_p \mathbf{m} = \nabla \cdot \left( \|\nabla \mathbf{m}\|^{p-2} \nabla \mathbf{m} \right)$ is called the $p$-Laplacian of $\mathbf{m}$. This partial differential equation is singular elliptic when $1 \le p < 2$ and degenerate elliptic when $p > 2$. When $p=2$, the problem reduces back to the typical elliptic equation.

Introducing the time variable and interpreting the evolution as a heat flow, we can determine this critical point using the gradient descent approach by solving the following evolution equation
\begin{equation}
\frac{\partial \mathbf{m}}{\partial t} = \Delta_p \mathbf{m} + \|\nabla \mathbf{m}\|^p \mathbf{m},
\label{Eqn:p-harmonicflow}
\end{equation}
with given boundary and initial conditions, where $\mathbf{m}: \Omega \times (0,T) \rightarrow \mathbb{S}^{n-1}$ contains both spatial and temporal components. The corresponding evolution is called the $p$-harmonic flow. The resulting equation has many important applications in science and engineering. For instance, we can treat the vector-valued function $\mathbf{m}$ as the so-called chromaticity of a colored image and use the $p$-harmonic flow for color image denoising \cite{tansapcas00,tansapcas01,vesosh02,lysoshtai04,golwenyin09}. The resulting gradient flow in micromagnetics leads to the Landau-Lifshitz equation, which is crucial in understanding nonequilibrium magnetism \cite{ewan00,leirun22}. When applied to the molecular orientation in liquid crystals, a related equation can be found in the Ericksen-Leslie continuum theory for nematic liquid crystals \cite{chkll87,linlus89,alo97}.

In this work, we rewrite the gradient flow equation (\ref{Eqn:p-harmonicflow}) as $\m_t = \m \times ( \m \times \Delta_p \m)$
and discretize the spatial discretization on the right-hand side of the equation using standard finite difference. The time evolution is then handled using the proposed STVDRK method. The numerical scheme is guaranteed to provide a solution $\m(t)\in\mathbb{S}^2$. We consider a discontinuous initial condition, which we define using two curves on the planes $x=\pm 1$. These curves, denoted by $z=h_{\pm}(y)$, are given by $h_+(y)=\pm 2 \sin(\pi y)$, where $y$ is in the interval $[-1,1]$. We then project these curves onto the unit sphere, represented by $\mathbb{S}^2$. 

In Figures \ref{Fig:DiscontinuousP2} and \ref{Fig:DiscontinuousP1}, we present the numerical solutions of the $p$-Laplacian flow with $p=2$ and $p=1$, respectively. In the first row of each figure, we plot the individual components of $\mathbf{m}(t,s)$ at two different time levels. The $x$-axis of each plot represents the \textit{spatial} parameterization $s$, normalized to the interval $[0,1]$. The second row contains the solution on the unit sphere. For the case when $p=2$, we observe that the discontinuity is smoothed out almost immediately, resulting in a smoothly varying curve on the sphere. On the other hand, considering the 1-harmonic flow, the evolution can maintain the sharp discontinuity {as shown in the $x$-component of the solution in Figure \ref{Fig:DiscontinuousP1}(a)}, while eliminating small fluctuations along the curve.

\section{{Conclusions}}

{We have developed a novel class of numerical integrators designed for solving ODEs on a unit sphere. Our approach mimics the widely-used TVDRK methods commonly used for dynamics in Cartesian space. We have made two key modifications: firstly, substituting the forward Euler step with the exponential map of the unit sphere, and secondly, replacing the averaging steps in the intermediate stages with SLERP interpolations. This approach opens up several directions for future work. One direction is the development of STVDRK schemes with numerical accuracy of fourth-order or beyond. While our current focus is on spherical geometry, it is also intriguing to consider numerical integrators on general surfaces or manifolds.}

\section*{Declarations}
Leung acknowledges the support of the Hong Kong RGC under grants 16302223 and 16300524.

\section*{Data availability}
Enquiries about data availability should be directed to the authors.

\bibliographystyle{plain}
\bibliography{syleung}

\section*{Appendix A: A Proof of $\left\|\p^{\tiny \mbox{RK2}} \right\|= 1 +O(h^4)$}

We assume that the evolution stays on a plane. By introducing the coordinates system in the plane that the evolution stays in, without loss of generality, assume the unit sphere is centered at the origin, and $\mathbf{p}^n$ is at $(0,-1)$. The first step of the TVDRK2 method arrives at $\p^{(1)}=\p^n+\mathbf{a}h$, where $\mathbf{a}=f(\p^n)$. We denote the second step by $\p^{(2)}=\p^{(1)}+\mathbf{b} h$, where $\mathbf{b}=f(\p^{(1)})$. Since the velocity is tangential to the sphere, we have $\p^n \perp \mathbf{a}$ and $\p^{(1)} \perp \mathbf{b}$. Finally, the TVDRK2 update is given by the average of $\p^n$ and $\p^{(2)}$.

\begin{figure}[!htb]
\centering
%(a)\includegraphics[trim=0 0 50 30, clip, width=0.45\textwidth]{figures/TVDRK2_1}
%(b)\includegraphics[trim=200 20 0 220, clip, width=0.35\textwidth]{figures/PTVDRK3_1}
\includegraphics[width=0.8\textwidth]{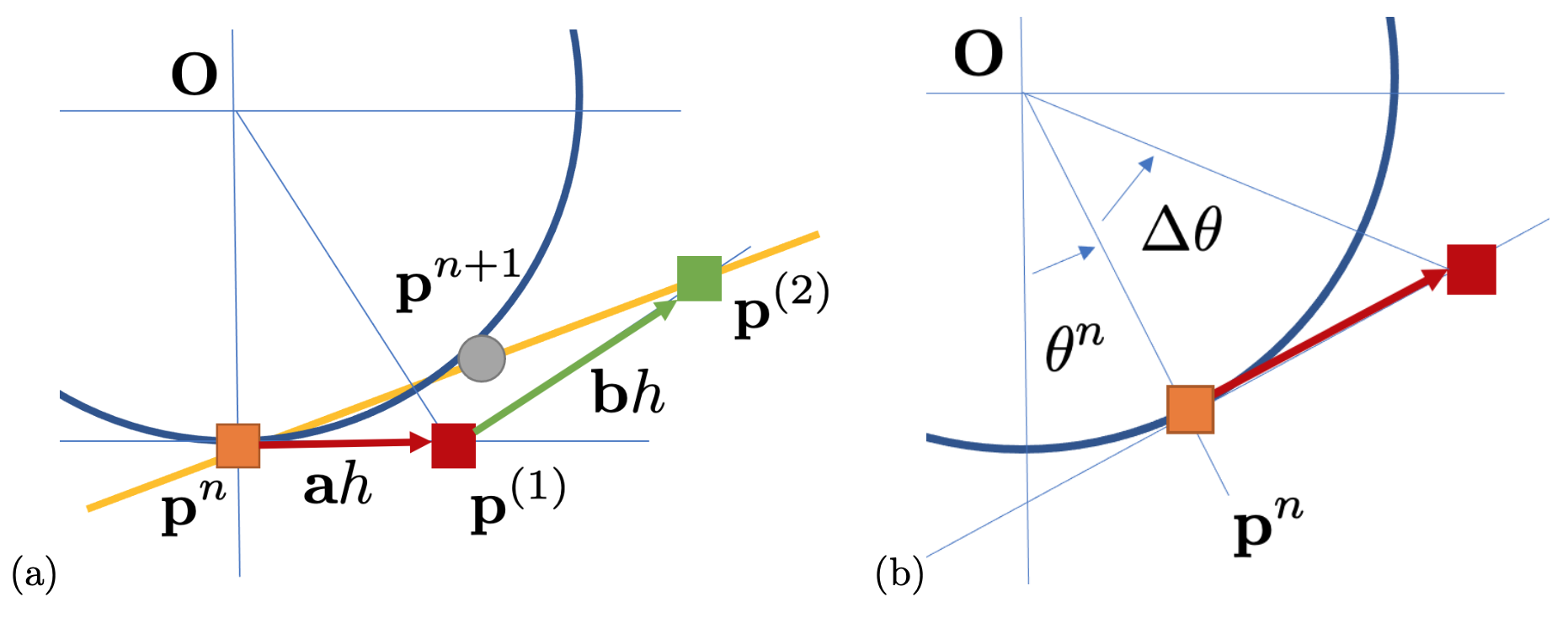}
\caption{Setting to show (a) $\left| \left\|\p^{\tiny \mbox{RK2}}\right\| - 1 \right|= O(h^4)$ and (b) PFE is first order accurate.} 
\label{Fig:TVDRK2_2}
\end{figure}

Denote $\Vert \mathbf{a} \Vert$ and $\Vert \mathbf{b} \Vert$ as $a$ and $b$. As illustrated in Figure \ref{Fig:TVDRK2_2}(a), the red vector can be expressed as $\mathbf{a}h = \langle ah,0 \rangle$. Consequently, $\mathbf{p}^{(1)}$ can be computed as $\mathbf{p}^{(1)} = \mathbf{p}^n+\mathbf{a}h=(ah,-1)$. Thus vector $\overrightarrow{Op^{(1)}}$ is determined as $\langle ah,-1 \rangle$. The green vector, that is $\mathbf{b}h$, is perpendicular to $\overrightarrow{Op^{(1)}}$ and has a positive $x$-component, hence it is pointing towards the direction of $\langle 1,ah \rangle$ . Together with the magnitude $bh$, the green vector can be represented as
$\mathbf{b}h = bh\frac{\langle 1,ah \rangle}{\sqrt{1+(ah)^2}}$. Then $\mathbf{p}^{(2)}$ can be computed as $\mathbf{p}^{(2)}=\mathbf{p}^{(1)}+\mathbf{b}h = \left(  ah+\frac{bh}{\sqrt{1+(ah)^2}}, -1+\frac{abh^2}{\sqrt{1+(ah)^2}} \right)$. Finally, $\mathbf{p}^{n+1}$ (i.e., the midpoint of $\overline{\mathbf{p}^n\mathbf{p}^{(2)}}$) can be computed as 
$
\mathbf{p}^{n+1} = \frac{1}{2}\left[\mathbf{p}^n+\mathbf{p}^{(2)}\right] =  \left(  \frac{ah}{2}+\frac{bh}{2\sqrt{1+(ah)^2}}, -1+\frac{abh^2}{2\sqrt{1+(ah)^2}} \right).
$
Now using Taylor's expansion, the norm of $\mathbf{p}^{n+1}$, can be computed as $\Vert \mathbf{p}^{n+1} \Vert =1+\frac{1}{8}(a-b)^2h^2-\frac{1}{128}\left( (a-b)^4+16a^3b \right)h^4 + O(h^6)$.
Since the velocity field $f$ is Lipschitz, we have $|a-b| = | \|f(\p^n)\|-\|f(\p^{(1)})\| | \le \|f(\p^n) - f(\p^{(1)})\| \le K \| \p^n - \p^{(1)}\| \le K h a =O(h)$. This estimate then implies $L=1+O(h^4)$. Since the local error in the norm is $O(h^4)$, we expect the global error in the solution norm is $O(h^3)$.

If not all intermediate steps are on the same plane, we can consider the plane passing through the center of the sphere, the point $\p^n$, and the second intermediate stage $\p^{(2)}$. We can replace the point $B$ in the above argument by projecting $\p^{(1)}$ onto the above plane. Then, a similar argument will follow.

\section*{Appendix B: A Proof of the Internal Projection Method PTVDRK3' is only Second-Order Accurate}

Assuming we have motion on the great circle of the unit sphere, the PTVDRK3' scheme involves intermediate stages that lie on the unit sphere. To represent the numerical solutions, we introduce the angle representation. Consider the velocity given by $\theta'=\theta$, which yields the exact solution $\theta(h)=\theta_0 \exp(h) = (1+h+\frac{1}{2}h^2+\frac{1}{6}h^3+\cdots) \theta_0$. The PFE step $\p^{n+1} = \mathcal{P}(\p^n + h f(\p^n,t^n))$ introduces a change in the polar angle denoted by $\Delta \theta = \tan^{-1}(h\theta^n)$, as shown in Figure \ref{Fig:TVDRK2_2}(b). This implies that $\theta^{n+1} = \theta^n + \tan^{-1}(h\theta^n) = \left(1 + h -\frac{1}{3} h^3 + \frac{1}{5} h^5 +\cdots \right) \theta^n$, making it a first-order scheme. Let $g$ be the expansion $g =1 + h -\frac{1}{3} h^3 + \frac{1}{5} h^5 +\cdots$. The PTVDRK2' scheme can be represented as $\theta_1=g \theta^n$, $\theta_2= g^2 \theta^n$, and therefore $\theta^{n+1}=\frac{1}{2} \left( 1+ g^2 \right) \theta^n = \left(1+h+\frac{1}{2} h^2 - \frac{1}{3} h^3+ \cdots \right) \theta^n$, indicating that PTVDRK2' is a second-order scheme. Finally, the PTVDRK3' scheme gives $\theta_1=g\theta^n$, $\theta_2= g^2 \theta^n$, $\theta_3 = \frac{1}{4} \left( 3+ g^2 \right) \theta^n$, $\theta_4 = \frac{1}{4} \left( 3+ g^2 \right) g \theta^n$. This implies $\theta^{n+1} =\left( 1+h+\frac{1}{2} h^2 - \frac{1}{6} h^3 -\cdots \right) \theta^n$. Since the sign of the third-order term is negative, the scheme is only second-order accurate.

\end{document}